\documentclass[oneside,reqno,11pt]{amsart}

\usepackage{amsthm}
\usepackage{amssymb}
\usepackage{url}
\usepackage{graphicx,xcolor,marginnote}

 \newtheorem{thrm}{Theorem}[]
 \newtheorem*{thrm*}{Theorem}
 \newtheorem{lemma}[thrm]{Lemma}
 \newtheorem*{lemma*}{Lemma}
 \newtheorem*{Vermutung*}{Vermutung}
 
 \newtheorem{prop}[thrm]{Proposition}
 \theoremstyle{definition}
 
\newtheorem*{defi*}{Definition}

  \theoremstyle{remark}

  \newtheorem{remark}[thrm]{Remark}
  \newtheorem*{remark*} {Remark}

\newcommand{\f}{\begin{align}}
\newcommand{\fo}{\begin{align*}}
\newcommand{\fe}{\end{align}}
\newcommand{\foe}{\end{align*}}
\newcommand{\map}[3] {#1:\, #2\, \longrightarrow \, #3}

\newcommand{\R}{\mathbb{R}}

\newcommand{\E}{\mathbb{E}}

\newcommand{\C}{\mathbb{C}}
\newcommand{\Tr}{\textup{Tr}}

\renewcommand{\a}{\alpha}
\renewcommand{\b}{\beta}
\newcommand{\g}{\gamma}
\renewcommand{\d}{\delta}
\newcommand{\e}{\eta}

\newcommand{\s}{\sigma}
\renewcommand{\l}{\lambda}

\renewcommand{\t}{\tau}

\renewcommand{\hat}{\widehat}

\renewcommand{\phi}{\varphi}
\renewcommand{\e}{\epsilon}
\renewcommand{\epsilon}{\varepsilon}

\newcommand{\lv}{\left\lvert}
\newcommand{\rv}{\right\rvert}

\newcommand{\lb}{\left(}
\newcommand{\rb}{\right)}
\newcommand{\lr}{\left[}
\newcommand{\rr}{\right]}

{\setlength\arraycolsep{2pt}

\newcommand{\lbb}{\big(}
\newcommand{\rbb}{\big)}

\renewcommand{\O}{\mathcal O}
\newcommand{\erfc}{\textup{erfc}}
\renewcommand{\ell}{\textup{ell}}
\newcommand{\tr}{\textup{Tr}^2}
\newcommand{\weak}{\operatorname{weak}}
\newcommand{\strong}{\operatorname{strong}}
\newcommand{\FT}{{\textup{FT}}}

\begin{document}
\bigskip

\title[]{Universality at weak and strong non-Hermiticity beyond the elliptic Ginibre ensemble}
\author[G. Akemann]{Gernot Akemann}
  \address{Faculty of Physics, Bielefeld University, P.O.Box 100131, D-33501 Bielefeld, Germany}
  \email{akemann@physik.uni-bielefeld.de}
  \author[M. Cikovic]{Milan Cikovic}
  \address{Faculty of Physics, Bielefeld University, P.O.Box 100131, D-33501 Bielefeld, Germany}
  \email{milan@physik.uni-bielefeld.de}
  \author[M. Venker]{Martin Venker}
  \address{Institut de Recherche en Math\'ematique et Physique, Universit\'e catholique de Louvain, Chemin du Cyclotron
  	2, Louvain-La-Neuve, Belgium}
  \email{martin.venker@uclouvain.be}


\begin{abstract}
% new version:
We consider non-Gaussian extensions of the elliptic Ginibre ensemble of 
complex non-Hermitian random matrices by fixing the trace $\Tr(XX^*)$ of the matrix $X$ with a hard or soft constraint.
These ensembles have correlated matrix entries and non-determinantal joint densities of the complex eigenvalues. 
We study global and local bulk statistics in these ensembles, in 
particular in the limit of weak non-Hermiticity introduced by Fyodorov, Khoruzhenko and Sommers. Here, the support of the limiting 
measure collapses to the real line.
This limit was motivated by physics applications and interpolates between the celebrated sine and Ginibre kernel.
Our results constitute a first proof of universality of the interpolating kernel.
Furthermore, in the limit of strong non-Hermiticity,
where the support of the limiting measure remains an ellipse,
we obtain local Ginibre statistics in the bulk of the spectrum. 
\end{abstract}

 \maketitle

%%%%%%%%%%%%%%%%%%%%%%%%%%%%%%%%%%%%%%%%%%%%%%%%%%%%%%%%%%%%%%%%%
\section{Introduction and Main Results}%\label{Sec-intro}

Despite its almost equally long history, the investigation of random matrices without symmetry
constraints is less advanced than that of random matrices which are for instance 
real symmetric, 
complex Hermitian or unitary. More generally, one distinguishes in Random Matrix Theory
(RMT) between matrix eigenvalues living on a one-dimensional set in the
complex plane $\C$ and matrices with ``genuinely complex'' eigenvalues. 
This distinction is 
somewhat imprecisely called Hermitian and non-Hermitian RMT. 
In physics applications the disctinction between Hermitian and non-Hermitian operators also plays an important role, where the 
latter are encountered e.g.~in quantum systems that are open or have a non-vanishing chemical potential. 
We refer to \cite{Handbook} for a list of modern applications in Hermitian and non-Hermitian physics problems.
This article focuses on the global and the local statistics in such non-Hermitian RMT where 
a transition between real and complex eigenvalues can be observed.

It is known that similar to Hermitian RMT, 
in the limit of large matrix size the
local statistics in non-Hermitian RMT fall
into different universality classes, depending on whether
the entries of the matrices are real, complex or quaternion. 
For example, the limiting local
correlations of eigenvalues in the bulk of the spectrum of an $N\times N$ random Gaussian matrix
with complex entries (the so-called Ginibre ensemble, a precise definition will be given below)
are given by those of the Ginibre point process. This point process is determinantal with the
so-called Ginibre kernel (see \eqref{kernel_strong} below) and may be seen as counterpart of the corresponding sine process
which describes the limiting local bulk statistics of random Hermitian matrices with complex
entries. 
Both
limiting point processes are highly universal
in the sense that they are limits of large classes
of different random matrix models.
The limit regime leading to the Ginibre kernel will be called limit of strong
non-Hermiticity, where we follow \cite{FKS}. The precise meaning will be explained below.

It appears not to be well known - at least in the mathematical community - that there
is another important limit regime in non-Hermitian RMT that leads to a one-parameter
family of limiting point processes \textit{interpolating} between the Ginibre and 
sine point process. This limit was found by Fyodorov, Khoruzhenko and Sommers \cite{FKS} and was coined weak non-Hermiticity limit. 
It occurs for random matrices that are almost Hermitian.
Surprisingly, it is precisely in this limit where a map between RMT and the corresponding effective field theory of the underlying 
physics problem can be made, for example in 
superconductors with columnar defects \cite{Efetov}
or in quantum chromodynamics with chemical potential \cite{SplittorffVerbaarschot}.
We refer to \cite{FS} for a review on this subject and further references to applications.
In \cite{AP14} a list of weakly non-Hermitian limiting kernels is given in different symmetry classes. These
correspond to local bulk, soft, or hard edge interpolating point processes that are all conjectured to be universal.

However, with
the exception of the soft edge limit of complex matrices \cite{Ben10}, there are no fully rigorous results 
proving the existence of these local 
weak non-Hermiticity
limits, let alone their possible universality (cf. Remark \ref{remark_weak} for
Ledoux's work \cite{Ledoux} on the crossover between the global limits semicircle and circular law).
In this article, we give a
first proof that for random
matrices with complex entries, the bulk limit of weak non-Hermiticity yields a universal point
process. We show this universality for two classes of ensembles. 
In the first ensemble that we call 
fixed trace elliptic ensemble, the norm $\Tr JJ^*$ of the matrix $J$ is fixed by a hard constraint. 
In our second ensemble the constraint is enforced more smoothly by adding a trace squared term to the density of the elliptic 
Ginibre ensemble.

In Hermitian RMT,
ensembles with a fixed trace constraint belong to the classical random matrix ensembles, going back
to 
\cite{Rosenzweig}. 
In analogy to  statistical mechanics, the fixed trace and standard Gaussian ensembles 
correspond to the microcanonical 
and canonical ensemble, resepectively. Despite this interpretation, 
to date nothing is known about the asymptotics when imposing a
fixed trace constraint in the non-Hermitian
Ginibre ensemble, in any limit regime.
The non-Gaussian nature of this ensemble makes it an ideal testing ground for
universality questions, in particular for 
the the weak non-Hermiticity limit.

Let us now make the previous more precise. The Ginibre ensemble
\cite{Ginibre}, which can be considered as a standard Gaussian measure, is defined as the probability 
measure on $\C^{N\times N}$ with density proportional to $\exp[-N\Tr(JJ^*)]$. Equivalently, a random matrix from the Ginibre 
ensemble can be realized 
as $J:=J_1+iJ_2$, where $J_{1,2}$ are independent Hermitian matrices from the Gaussian Unitary Ensemble (GUE), i.e.~the matrix 
distribution on the space of $N\times N$ Hermitian matrices with density proportional to $\exp[-N\Tr(J_{1,2}^2)]$.
The matrix $J$ will (almost surely) not be unitarily diagonalizable, but a Schur decomposition can be used to obtain the joint 
density of its eigenvalues $z_1,\dots,z_N$ on $\C^N$ as proportional to
\begin{align}
	 \exp\lr{\frac{\beta}{2}\sum_{j\not=l}\log\lv z_j-z_l\rv-N\sum_{j=1}^N\lv z_j\rv^2}\rr, \ \ \mbox{for}\ \ 
\beta=2.\label{Coulomb}
\end{align}
The eigenvalues form a two-dimensional Coulomb gas which represents another, well-known physics application of non-Hermitian RMT. 
Here, the eigenvalues $z_j$ correspond to charged point particles, at inverse temperature $\beta$. Only at the specific temperature 
$\beta=2$ the point process is determinantal, 
and
the eigenvalue statistics can be efficiently analyzed, 
i.e. by the fact that its correlation functions are given as determinants of  
a kernel $K_N$, which in turn can be studied using orthogonal polynomials in  the complex plane. Doing so, it was found that in the 
large $N$ limit, the eigenvalues of $J$ fill the unit disc in the complex plane with uniform density (the so-called ``circular 
law''). Also the limiting local correlations between close eigenvalues 
in the bulk
could be computed 
(see \eqref{kernel_strong}) \cite{Ginibre}.

Another important and well-studied model is the elliptic (Ginibre) ensemble. It was introduced as an interpolation between 
Hermitian and non-Hermitian matrices  by taking 
$J:=\sqrt{1+\t}J_1+i\sqrt{1-\t}J_2$, where $J_{1,2}$ are again independent GUE matrices and $\t\in(-1,1)$ controlls the degree of 
(non-)Hermiticity. The Ginibre ensemble is recovered choosing $\t=0$ and the GUE is obtained by formally taking the limit $\t\to1$. The 
elliptic ensemble has the density
\begin{align}
  P_{N,\ell}(J):=\frac1{Z_{N,\ell}}\exp\left[-\frac N{1-\t^2}\Tr\left(JJ^*-\frac\t2(J^2+{J^*}^2)\right)\right]\label{elliptic}
 \end{align}
on $\C^{N\times N}$,
where  $Z_{N,\ell}$ is the normalizing constant. The reader will note the resemblence with the bivariate normal distribution. In 
this interpretation, $\t$ is the correlation coefficient between $\Re J_{j,l}$ and $\Re J_{l,j}$ and $-\t$ the one between $\Im 
J_{j,l}$ and 
$\Im J_{l,j}$, for $j\not=l$. 
 The eigenvalue distribution corresponding to \eqref{elliptic} is again determinantal, known in closed form (cf. 
\eqref{EV_density_epsilon}) and has been analyzed in great detail, cf. \cite[Chapter 18]{Handbook} and references therein. As 
$N\to\infty$, 
and $|\tau|<1$ is fixed, its eigenvalues spread  uniformly in the set 
\begin{align}
	E:=\left\{Z\,:\,\lb\frac{\Re Z}{1+\t}\rb^2+\lb\frac{\Im Z}{1-\t}\rb^2\leq1\right\},\label{ellipse}
\end{align}
a fact that is termed ``elliptic law''. 
Its limiting local correlations, however, coincide for $\t\in(-1,1)$ fixed with those of 
the Ginibre ensemble \cite{FKS}. 

In this article, we study fixed trace versions of the Ginibre and elliptic ensembles. Formally, the fixed trace Ginibre ensemble $P_N$ can be seen as 
\begin{align}
P_{N}(dJ)=\frac 1{Z_N}\d(NK_p-\Tr JJ^*)dJ \label{FT_formally}
\end{align}
with $\d$ denoting the so-called Dirac delta function, $Z_N$ the normalizing constant and $K_p>0$ being an arbitrary constant, at 
which the 
%``energy'' 
normalized norm
$\Tr JJ^*/N$ is fixed. A rigorous definition is as follows:
Define
\begin{align*}
\mathcal S_N:=\{J\in\C^{N\times N}\,:\,\Tr JJ^*=NK_p\}
\end{align*}
with $K_p>0$ arbitrary. $\mathcal{S}_N$ is the sphere of radius $NK_p$ in the vector space $\R^{2N^2}\hat{=} \, \C^{N\times N}$. It 
is a well-known fact that there is a unique probability measure $P_N$ on $\mathcal S_N$ that is invariant under the orthogonal group 
acting on the $2N^2$-dimensional sphere. We will call $P_N$ the \textit{fixed trace Ginibre ensemble}. In the viewpoint of 
statistical mechanics,
the fixed trace Ginibre ensemble corresponds to the microcanonical ensemble, whereas the Ginibre ensemble corresponds to the 
canonical ensemble.

More generally, we define the \textit{elliptic fixed trace ensemble} for $\t\in(-1,1)$ as the probability measure 
\begin{align}
P_{N,\FT}(dJ):=\frac1{Z_{N,\FT}}\exp\lr{\frac {\t N}{2(1-\t^2)}\Tr\left(J^2+{J^*}^2\right)}\rr P_N(dJ),\label{FT}
\end{align}
where $Z_{N,\FT}$ denotes the normalization constant. In the special case $\t=0$, we recover the fixed trace Ginibre ensemble 
$P_N$. 
Note that the density of $P_{N,\FT}$ w.r.t.~$P_N$ is proportional to the density of the elliptic ensemble 
\eqref{elliptic} on $\mathcal S_N$ w.r.t.~the Lebesgue measure, since on $\mathcal S_N$ the term $\Tr JJ^*$ is constant. 

Furthermore, we consider an interpolation between \eqref{elliptic} and \eqref{FT}. 
 This ensemble is given by the density on $\C^{N\times N}$ 
\begin{align}
P_{N,\tr}(J):=\frac1{Z_{N,\tr}}\exp\lr{-\frac N{1-\t^2}\Tr\left(JJ^*-\frac\t2(J^2+{J^*}^2)\right)-\g 
(\Tr JJ^*-NK_p)^2}\rr,\label{def_matrixensemble}
\end{align}
where $\g\geq0,K_p\in\R$ are fixed, $\t\in(-1,1)$ and $Z_{N,\tr}$ is the normalization constant. This model is of the form \eqref{FT}  with a Gaussian approximation of the delta function in \eqref{FT_formally} penalizing deviations of $\Tr JJ^*$ from the value $NK_p$. If the strength of 
the penalization $\g$ goes to infinity, we have $P_{N,\tr}(dJ)\to P_{N,\FT}(dJ)$ in distribution. We will call 
\eqref{def_matrixensemble} the \textit{trace-squared ensemble}.
$P_{N,\tr}$ rather puts a ``soft constraint'' on $\Tr JJ^*$, whereas $P_{N,\FT}$ is obtained using a ``hard constraint''. In 
contrary to the elliptic ensemble, $P_{N,\tr}$ for $\g>0$ and $P_{N,\FT}$ are non-Gaussian and non-determinantal (cf. Remark 
\ref{non-determinantal} below).

Let us briefly comment on  fixed trace ensembles in Hermitian RMT. They are classical random matrix ensembles
introduced in \cite{Rosenzweig} (as reported in \cite[Chapter 27]{Mehta}). While for Hermitian models the limiting global 
correlations have  been known  since \cite{Rosenzweig,Mehta} to be given by the semicircle distribution for all three symmetry 
classes, a rigorous proof of universality of the sine kernel for a fixed trace model of Hermitian matrices has only been presented 
much more recently in \cite{GoetzeGordin}, cf.~\cite{AV} for earlier heuristic arguments. 
Local statistics at the soft and hard edge have been analyzed for fixed trace $\b$-ensembles in \cite{ZLQ,CLZ,LZ}.
More general fixed 
trace models have been considered	
where the trace of a polynomial in the random matrix is fixed \cite{ACMV99}.
This corresponds to the microcanonical version of non-Gaussian generalizations of the classical Wigner-Dyson enesmbles. Here, on the 
global scale the 
non-universality of higher order connected correlation functions (cluster functions) was argued for using loop equations.
 
For non-Hermitian fixed trace models, almost nothing is known with the exception of a formula for the
spectral density of complex eigenvalues at finite $N$ in the fixed trace Ginibre ensemble \cite{DLC00}. In particular, even the 
limiting measure for this simplest possible fixed trace model is unknown, let alone the local correlations.

Furthermore, the models \eqref{FT} and \eqref{def_matrixensemble} provide an excellent testing ground for studying the so-called 
limit of weak non-Hermiticity. The Ginibre correlations arise in the  situation of \textit{strong non-Hermiticity}, meaning that the 
anti-Hermitian part of the random matrix is of the same order in $N$ as the Hermitian part. This results in a limiting global 
distribution (the weak limit of the empirical spectral distribution $N^{-1}\sum_{j=1}^N \d_{z_j}$, the $z_j$'s being the 
eigenvalues) that is supported on a two-dimensional subset of $\C$.

In contrast to the limit of strong non-Hermiticity, the limit of \textit{weak non-Hermitici\-ty} describes a situation where the limiting global distribution of the non-Hermitian random matrix is supported on the real line but the local correlations still extend to the complex plane. For the elliptic ensemble, this happens if the parameter $\t$ is chosen as $\t=1-\kappa/N$ for some $\kappa>0$ not depending on $N$ (see \eqref{ellipse}) and thus the random matrix is almost Hermitian.
The limit of weak non-Hermiticity was first discussed perturbatively  in \cite{FKS97a}. The limiting point process is determinantal 
and its kernel was derived in terms of Hermite polynomials in the complex plane in \cite{FKS97b}, see \cite{FKS} for details. The 
limit of weak non-Hermiticity allows to describe the transition between sine and Ginibre kernel. This makes its universality highly 
suggestive, but so far only heuristic arguments in favour of this conjecture exist. For matrix ensembles with independent entries, 
these arguments can be found in \cite{FKS}, whereas more general ensembles (see \eqref{general_density} below) were treated in 
\cite{Ake02}.

Let us now turn to the main results of this paper.
Global and local statistics are usually studied using correlation functions. For a probability density $P$ on $\C^N$ and $1\leq k\leq N$, the $k$-th correlation function is defined as
\begin{align}
	\rho^k(z_1,\dots,z_k):=\frac{N!}{(N-k)!}\int_{\C^{N-k}}P(z_1,\dots,z_N)dz_{k+1}\dots dz_N.\label{def_correlation_functions}
\end{align}
The correlation functions are multiples of the marginal densities. 
Let $\rho_{N,\tr}^{k}$ denote the $k$-th correlation function of the eigenvalue density corresponding to $P_{N,\tr}$. Although 
$P_{N,\FT}(dJ)$ is not absolutely continuous, its eigenvalue distribution is for $N\geq2$. This is due to the fact that the 
constraint $\Tr JJ^*=NK_p$ translates to $\sum_{j=1}^N\lv z_j\rv^2+\sum_{i<j}^N\lv T_{ij}\rv^2=NK_p$, where $z_1,\dots,z_N$ are  the 
eigenvalues 
collected in the diagonal matrix $Z$,
and the 
$T_{ij}$ are complex variables stemming from the Schur decomposition $J=U(Z+T)U^*$.
Thus the constraint is somewhat less 
restrictive on the spectral level. For $N=1$, $P_{N,\FT}$ coincides with its eigenvalue distribution and is a probability 
distribution on a dilation of the unit circle. We will throughout the paper tacitly assume $N\geq2$ when speaking of correlation 
functions. Let $\rho_{N,\FT}^k$ denote the $k$-th correlation function of the eigenvalue density corresponding to $P_{N,\FT}$. An 
integral representation of $\rho_{N,\FT}^k$ for $k<N$ will be derived in Lemma \ref{Lemma_FT}.
We are now ready to state our main results. For the sake of brevity, we consider both types of ensembles, \eqref{FT} and \eqref{def_matrixensemble}, simultaneously.  We start with the strongly non-Hermitian situation.

\begin{thrm}[Limit of strong non-Hermiticity]\label{theorem_strong}
	Let $\t\in(-1,1)$ be fixed. Let $\rho_N^k$ denote either $\rho_{N,\FT}^k$ or $\rho_{N,\tr}^k$. Then there are constants 
$c_{1},c_2,C>0$, depending only on $K_p,\t$ and in the case $\rho_N^k=\rho_{N,\tr}^k$ also on $\g$, such that with $E:=\{Z\,:\,c_1 
(\Re 
Z)^2+c_2(\Im Z)^2\leq 1\}$ the following holds:
	\begin{enumerate}
	\item For any $Z\in\C$, $Z\notin\partial E$, we have
	\begin{align*}
	\lim_{N\to\infty}\frac{1}{N}\rho_{N}^1(Z)=1_{E^\circ}(Z)\cdot\frac C\pi.
	\end{align*}
	\item For $k=1,2\dots$, $Z\in E^\circ$, $z_1,\dots,z_k\in\C$, as $N\to\infty$
\begin{align*}
 &\frac{1}{(CN)^k}\rho_{N}^{k}\lb Z+\frac{z_1}{\sqrt{CN}},\dots,Z+\frac{z_k}{\sqrt{CN}}\rb=\det\lb K_{\strong}(z_j,z_l)\rb_{j,l\leq k}+\O\lb\frac{1}{\sqrt{N}}\rb,
\end{align*}
where 
\begin{align}\label{kernel_strong}
K_{\strong}(z_j,z_l):=\frac{1}{\pi}\exp\lr{-\frac{\lv z_j\rv^2+\lv z_l\rv^2}2+z_j\overline{z_l}}\rr.
\end{align}
The $\O$ term is uniform for $Z$ from any compact subset of $E^\circ$ and $z_1,\dots,z_k$ from compacts of $\C$.
\end{enumerate}
\end{thrm}

\begin{remark}%\label{remark_normal}
\noindent
\begin{enumerate}
\item
The previous theorem shows in particular that the elliptic fixed trace ensemble \eqref{FT} and trace-squared ensemble \eqref{def_matrixensemble} belong to the same strongly non-Hermitian bulk universality class as the Ginibre ensemble. 
\item For the trace-squared ensemble, the constant $C$ and the elliptic set $E$ are given in Proposition \ref{strong_linearized}. The set $E$ differs from the limiting support of the elliptic ensemble, except if $K_p=1$ or $\g=0$, in which case $C=(1-\t^2)^{-1}$.
\item In \cite{TV15} for the Ginibre ensemble and \cite{Lee-Riser} for the elliptic ensemble, the convergence is shown to be exponentially fast. In view of these results, it is likely that the bounds on the rate of convergence in Theorem \ref{theorem_strong} can be improved. As this is not one of the main purposes of this work, we will not pursue this here.
\item Of course, the models \eqref{FT} and \eqref{def_matrixensemble} can also be considered on the set of $N\times N$ normal 
matrices. Theorem \ref{theorem_strong} and Theorem \ref{theorem_weak} below extend to this situation as well. Note that when 
considered as a normal matrix model, the joint density of the eigenvalues of the trace-squared ensemble is proportional to
\begin{align*}
	   \prod_{j<l}\lv z_j-z_l\rv^2\exp\lr-\frac{N}{1-\t^2}\lb\sum_{j=1}^N \lv z_j\rv^2-\frac{\t}{2}(\sum_{j=1}^N 
	   	z_j^2+\overline{z_j}^2)\rb-\g\lb\sum_{j=1}^N\lv z_j\rv^2-NK_p\rb^2\rr.
\end{align*}
When considered on $\C^{N\times N}$, the model does not have the same eigenvalue distribution. Note further that when considering 
a fixed trace model of normal matrices, its eigenvalue distribution is not absolutely continuous, since the constraint $\Tr 
JJ^*=NK_p$ implies $\sum_{j=1}^N\lv z_j\rv^2=NK_p$. Correlation functions $\rho_N^k$ defined as marginal densities therefore exist 
for $k<N$ only. They can be defined via the relation
	\begin{align}
	\int f(z_1,\dots,z_k)\rho_N^k(z_1,\dots,z_k)dz=\E_N\sum_{j_1\not=j_2\not=\dots\not=j_k}^Nf(z_{j_1},\dots,z_{j_k}),\label{characterization_corr_functions}
	\end{align}
where $f$ is a test function and $\E_N$ denotes expectation w.r.t.~the fixed trace (eigenvalue) measure.
\end{enumerate}
\end{remark}
Let us also give an overview over results in the literature on universality of the kernel $K_{\strong}$. In \cite{TV15}, the Ginibre kernel was shown for the non-Hermitian analogs of Wigner matrices, i.e.~random matrices with independent entries (and with independent real and imaginary parts), that have exponentially decaying distributions and fulfill certain moment conditions. Due to the lack of spectral calculus for matrices without any symmetries, there is no clear non-Hermitian analog of the rich class of unitary invariant ensembles in Hermitian RMT. In general, there is a conflict between having a normalizable density (i.e.~a well-defined model) and having an eigenvalue distribution in closed form.

One class of models in the literature is given by densities proportional to
\begin{align}
	\exp[-\s N\Tr(JJ^*)+\Re \Tr(\Phi(J))],\label{general_density}
\end{align}
where $\s>0$ and $\Phi$ is a potential. The ensembles \eqref{general_density} are determinantal, but over $\C^{N\times N}$ the density is normalizable only for very specific potentials $\Phi$. Basically, these are either $\Phi(J)=J^2$, corresponding to the elliptic ensemble, or logarithmic with special coefficient, see the discussion in \cite[Chapter 39]{Handbook}. This problem can be circumvented by simply considering the density on a sufficiently large compact set instead of $\C^{N\times N}$ (\cite{ElbauFelder}, see also \cite{BleherKuijlaars}). However, local universality has so far not been shown for these truncated ensembles.

A common way around these difficulties is to consider normal matrix models. For normal matrices, the spectral calculus allows to define unitary invariant ensembles, e.g.~with densities proportional to $\exp[-N\Tr V(JJ^*)]$ for some potential $\map{V}{\C}{\R}$ of sufficient growth at infinity. These models have determinantal eigenvalue distributions and belong to the same (bulk) universality class as the Ginibre ensemble \cite{AHM11,Ber08}. \\

We continue with the weakly non-Hermitian situation. Note the slightly differing values of $\t$ in parts a) and b) of the following theorem.

\begin{thrm}[Limit of weak non-Hermiticity]\label{theorem_weak}
Let $\rho_N^k$ denote either $\rho_{N,\FT}^k$ or $\rho_{N,\tr}^k$.
There is a constant $C>0$, depending only on $K_p$ and in the case $\rho_N^k=\rho_{N,\tr}^k$ also on $\g$, such that the following 
holds: 
\begin{enumerate}
\item Let $\t=\t_N=1-\frac{\kappa}{N}$ with $\kappa>0$ fixed. Then for any $Z\in\C\setminus\R$
\begin{align}
	\lim_{N\to\infty}\frac{1}{N}\rho_{N}^1(Z)=0,\label{density_weak_1}
\end{align}
and for any $X\in\R$
\begin{align}
\lim_{N\to\infty}\int_{\R}\frac{1}{N}\rho_{N}^1(X+iY)dY=\frac C{2\pi}\sqrt{\frac4C-X^2}\,1_{[-\frac{2}{\sqrt{C}},\frac2{\sqrt{C}}]}(X).\label{density_weak_2}
\end{align}
\item Set $\nu(X):=\frac{C}{2\pi}\sqrt{\frac{4}{C}-X^2}$ and $\t=\t_N:=1-\dfrac{\a^2}{2N\nu(X)^2}$, $\a>0$. Then for $k=1,2\dots$, as $N\to\infty$
	\begin{align*}
		&\frac{1}{\lb N\nu(X)\rb^{2k}}		\rho_{N}^{k}\lb X+\frac{z_1}{N\nu(X)},\dots,X+\frac{z_k}{N\nu(X)}  \rb\\&=\det\lb K_{\weak}\lb 
z_j,z_l\rb\rb_{j,l=1,\dots k}+\O\lb \frac{\log N}{\sqrt{N}}\rb,
\end{align*}
		where we denote $z_j:=x_j+iy_j$ and define
		\begin{align*}
		&K_{\weak}\lb z_1,z_2\rb:=\frac{\sqrt{2}}{\sqrt{\pi}\a}\exp\lr-\frac{
			y_1^2+y_2^2}{ \a^2 }\rr\frac{1}{2\pi}
		\int_{-\pi}^{\pi}\exp\lr-\frac{\a^2u^2}{2}+iu(z_1-\bar{z}_2)\rr
		du.
	\end{align*}
	The $\O$ term is uniform for $X\in (-\frac{2}{\sqrt{C}}+\d,\frac2{\sqrt{C}}-\d)$ for any $\d>0$ fixed and 
any $z_j, j=1,\dots,k$ chosen from an arbitrary compact subset of $\C$.
\end{enumerate}
\end{thrm}

\begin{remark}\label{remark_weak}\noindent
\begin{enumerate}
	\item Part a) of the previous theorem shows that the support of the limiting measure collapses to the real axis whenever $\t$ is in any $1/N$ neighborhood of $1$. Moreover, the limiting marginal density of the real part $X$ is the semicircle density $\nu(X)$, which is used in part b) to rescale not only the $z_j$'s, but also $\t$, in order to make the limiting kernel independent of $X$.
	\item $K_{\weak}$ agrees with the kernel given in \cite{FKS}, where $x_1=-x_2$ was chosen. We have simply rescaled $\a$ to 
make the kernel independent of $X$. The 
independence of $K_{\weak}$ on $\g$ and $K_p$ constitutes 
our universality result for
the non-Gaussian ensembles \eqref{FT} and \eqref{def_matrixensemble}
in the limit of weak non-Hermiticity.
	\item It is not hard to show that in the limit $\a\to0$ we obtain
	\begin{align*}
	&\lim_{\a\to0}\int_{\C^k
}f(z_1,\dots,z_k)\det\lb K_{\weak}\lb 
z_j,z_l\rb\rb_{j,l=1,\dots,k}dz_1\dots dz_k\\
&=\int_{\R^k} f(x_1,\dots,x_k)\det\lb \frac{\sin(\pi(x_j-x_l))}{\pi(x_j-x_l)}\rb_{j,l=1,\dots,k}dx_1\dots dx_k
	\end{align*}
	for any bounded and continuous function $\map{f}{\C^k}{\R}$ of bounded support.
	
	For the limit $\a\to\infty$, we need to rescale the variables $z_j$ with $\a$ in order to account for the difference in the local scales (in $N$) in the limits of weak and strong non-Hermiticity. Here it is straightforward to get
\begin{align*}
		&\lim_{\a\to\infty}\det\lb\a^2 K_{\weak}\lb \a z_j,\a z_l\rb\rb_{j,l=1,\dots,k}=\det\lb K_{\strong}\lb z_j,z_l\rb\rb_{j,l=1,\dots,k}.
	\end{align*}	
	\item In \cite{Ledoux}, the first correlation function of the elliptic ensemble $\rho^1_{N,\ell}(X+\frac{iy}{N})$ was studied in the limit of weak non-Hermiticity. Here, the variable on the real axis is in the global scaling regime whereas the imaginary variable is in the local scaling regime. With this particular choice of variables, one sees a transition from the semicircle law to the circular law.
	\item For the trace-squared ensemble, the constant $C$ in Theorem \ref{theorem_weak} is given by 
\begin{align*}
 C=\frac1{2}-2\g K_p+\frac12\sqrt{16\g^2 K_p^2-8\g K_p+16\g+1}.
\end{align*}
In particular, if $\g=0$ or $K_p=1$, then $C=1$. For the elliptic fixed trace ensemble, $C$ is given as $C=K_p^{-1}$.
\end{enumerate}
\end{remark}

Let us 
add some further
remarks. It was found in \cite{Ben10} that a similar weakly non-Hermitian scaling limit can be defined at the edge of the spectrum. 
There, for the largest real eigenvalue an interpolation between the Tracy-Widom and Gumbel distribution was found, where the latter 
corresponds to the Fredholm determinant of the complementary error function kernel of the Ginibre ensemble. The same interpolating 
kernel was found in a chiral variant of the Ginibre ensemble \cite{ABen10} and it was shown to be a one-parameter deformation of the 
Airy kernel.
We refer to \cite{AP14} for an entire list of one-parameter interpolating kernels at the edge, in the bulk and at the origin,
corresponding to the limits in Ginibre ensembles and its chiral partners with real, complex or quaternion matrix elements, see also  
references therein.

We have already mentioned the 2D-Coulomb gases of the complex eigenvalues in \eqref{Coulomb}. Without going into the vast 
literature, let us mention that recent results on 
fluctuations of such complex $\beta$-ensembles around their macroscopic limits have been considered up to the finest possible scale 
in \cite{Leble15,LebleSerfaty16}. Fluctuations of linear statistics have been studied in \cite{BBNY16}, whereas edge universality 
has been shown in \cite{Chafaietal13}.

Let us finish this section with some comments on our methods of proofs and organization of the article. The general idea is to 
understand the ensembles under consideration as averages over perturbations of simpler ensembles. These simpler ensembles will turn 
out to 
be complex-valued measures with determinantal eigenvalue distributions (also in the sense of complex-valued measures) that can be 
analyzed using a double contour integral formula.  Generally speaking, the asymptotics of the trace-squared ensemble 
\eqref{def_matrixensemble} 
are
easier to analyze as the model is absolutely continuous.
Its analysis will turn out to be useful for studying the fixed trace models as well. Recall that $P_{N,\tr}(J)$ is proportional to 
$P_{N,\ell}(J)\exp\lr -\g(\Tr JJ^*-NK_p)^2\rr$, where $P_{N,\ell}$ has been defined in \eqref{elliptic}. If $\Tr JJ^*$ is 
sufficiently well concentrated around $NK_p$, then the 
trace-squared term $\g(\Tr JJ^*-NK_p)^2$ can be expected to be of order 1 and thus might not have much influence at least on the 
global limit of the model. The square in this term can then be linearized at the expense of an additional integral by understanding 
$\exp\lr -\g(\Tr JJ^*-NK_p)^2\rr$ as the Fourier transform of a Gaussian random variable. This gives rise to a Gaussian average over 
elliptic ensembles, perturbed by a term of the form $\exp\lr it\Tr JJ^*\rr$ (with integration variable $t$), and to an expression of 
$P_{N,\tr}$ as average over linearized ensembles (see \eqref{linearization_full}). A similar expression of $P_{N,\FT}$ can not 
possibly 
hold, but from the expression of $P_{N,\tr}$ we can conclude a related formula for the correlation functions of $P_{N,\FT}$ (see 
Lemma \ref{Lemma_FT}). 

However, except for the special cases $\g=0$ or $K_p=1$, $\Tr JJ^*$ will not be close to $NK_p$ and thus the statistic $\g(\Tr 
JJ^*-NK_p)^2$ will not be small. In this case a necessary recentering needs to be executed first. This and the linearization for the 
trace-squared ensemble are done in Section \ref{Sec-linear}. The asymptotics of the linearized ensemble are then derived in Section 
\ref{section_soft}. Using these results, a proof of the main theorems for the trace-squared ensemble is given in Section 
\ref{Sec-proof}. The significantly more complex analysis of the fixed trace ensemble is done in Section \ref{Sec_proof_FT}.\\

{\sc Acknowledgements:} The work was supported by the German research council DFG through grants AK35/2-1 (G.A.), 
IGK "Stochastics and Real World Models" Beijing--Bielefeld (M.C.), and 
CRC 701 "Spectral Structures and Topological Methods in Mathematics" (M.V.), as well as by the European Research Council under the European Unions Seventh
Framework Programme (FP/2007/2013)/ ERC Grant Agreement n.  307074 (M.V.).

%%%%%%%%%%%%%%%%%%%%%%%%%%%%%%%%%%%%%%%%%%%%%%%%%%%%%%%%%%%%%%%%%
\section{Recentering and Linearization}\label{Sec-linear}
We will start our analysis of the ensemble \eqref{def_matrixensemble} by a recentering of the density which takes into 
account a part of $\g 
(\Tr JJ^*-NK_p)^2$ that has an influence on the global distribution of the eigenvalues. It is not hard to see 	that 
$\lim_{N\to\infty}\E_{N,\ell}\Tr(JJ^*)/N=1$, where $\E_{N,\ell}$ denotes expectation w.r.t.~the elliptic 
ensemble \eqref{elliptic}. Hence for $K_p=1$ the trace constraint does not change much and the global asymptotics (such as the support of 
the limiting 
ellipse) of the ensembles $P_{N,\tr}$ and \eqref{elliptic} should coincide. Thus in this case no recentering 
should be necessary as the penalization just reinforces the convergence of $N^{-1}\Tr JJ^*$ to $K_p$.
In contrast, choosing $K_p$ different from 
$1$ enforces $N^{-1}\Tr JJ^*$ to have a different limit, which should lie between $K_p$ and $1$. In this case it will be 
convenient to renormalize the ensemble. To this end, let us for $a>b\geq0$ consider the family of densities 
$P_{a,b}$ with
\begin{align}
 P_{a,b}(J):=\frac1{Z_{a,b}}\exp\lr{-a\Tr JJ^*+\frac b2\Tr \lb J^2+{J^*}^2\rb}\rr\label{K_density}
\end{align}
and normalization constant $Z_{a,b}$.
The motivation for introducing \eqref{K_density} comes from the following manipulation of $P_{N,\tr}$. Let $K>-(2\g(1+\t))^{-1}$ and rewrite
\begin{align}
 &P_{N,\tr}(J)\notag\\
 &=\frac{\exp\lr{-N\lb\frac {1}{1-\t^2}+2\g K\rb\Tr JJ^*+\frac{\t 
N}{2(1-\t^2)}\Tr \lb J^2+{J^*}^2\rb-\g 
(\Tr JJ^*-N(K_p+K))^2}\rr}{Z_{N,\tr}\exp\lr{-\g N^2(K^2+2KK_p)}\rr}\label{EV_density}\\
&=\frac{Z_{a,b}}{Z_{N,\tr}\exp\lr{-\g N^2(K^2+2KK_p)}\rr}P_{a,b}(J)\exp\lr{-\g 
(\Tr JJ^*-N(K_p+K))^2}\rr\notag
\end{align}
with 
\begin{align*}
	 a:=N\lb\frac {1}{1-\t^2}+2\g 
	 K\rb,\qquad b:=\frac{\t N}{1-\t^2}.
\end{align*}
For the rest of the paper, we stick to this choice of $a$ and $b$. 
Note that the condition \newline$K>-(2\g(1+\t))^{-1}$ ensures that $a>b$, and thus the normalizability of $P_{a,b}$.
Furthermore, comparing \eqref{EV_density} and \eqref{K_density} shows that
\begin{align}
	\frac{Z_{N,\tr}\exp\lr{-\g N^2(K^2+2KK_p)}\rr}{Z_{a,b}}=\E_{a,b}\exp\lr{-\g 
		(\Tr \tilde{J}\tilde{J}^*-N(K_p+K))^2}\rr,\label{Laplace}
\end{align}
where $\E_{a,b}$ denotes expectation w.r.t. $P_{a,b}$ and we use the following convention throughout the paper: In equations, the matrix $\tilde{J}$ is an integration variable, whereas $J$ denotes an arbitrary, but fixed matrix. However, as there is no ambiguity, we will also let $J$ denote the random matrix associated to a specified ensemble.

The following lemma gives the optimal choice of the so far arbitrary $K$ in \eqref{EV_density}. In short, we determine $K$ such that the statistic $\Tr JJ^*-N(K_p+K)$ is concentrated under $P_{a,b}$.
\begin{lemma}\label{Lemma_K}
	\noindent
	\begin{enumerate}
		\item
 For each $\t\in(-1,1)$, $\g\geq0$ and $K_p\in\R$, there is a unique $K=K(\t)=K(\t,\g,K_p)$ such that for some constants $C_1,C_2$ 
independent of $\t$ and $N$, we 
have for all $N$ 	
\begin{align}
 0<C_1\leq\E_{a,b}\exp\lr{-\g 
 	(\Tr \tilde{J}\tilde{J}^*-N(K_p+K))^2}\rr\leq C_2<\infty.\label{bounds}
\end{align}
\item
$K$ of a) is the unique positive zero of the cubic equation \eqref{cubic}. The limit \newline$\bar{K}:=\lim_{\t\to1}K(\t)$ exists,
\begin{align*}
 \bar{K}>-\frac{1}{4\g}
\end{align*}
and
\begin{align}
 \lv K(\t)- \bar{K}\rv=\O(\lv \t-1\rv)\text{ as }\t\to1.\label{K_derivative}
\end{align}
\item
If $\lv\t-1\rv=\O(1/N)$, \eqref{bounds} holds with $K$ replaced by $\bar{K}$.
\item $K$ is continuous in $\g$ and the limit $K_\FT:=\lim_{\g\to\infty}\g K$ exists and is larger than $\max\{-(2(1-\t))^{-1},-(2(1+\t))^{-1}\}$. Furthermore, $\bar{K}_\FT:=\lim_{\t\to1}K_\FT=\frac{K_p^{-1}-1}{4}$ and $\lv K_\FT-\bar{K}_\FT\rv=\O(\lv\t-1\rv)$.
\end{enumerate}
\end{lemma}

\begin{proof}

It is easy to check that under $P_{a,b}$, the matrix elements $\{\Re 
J_{j,k}, \Im J_{j,k}\}_{j,k}$  are jointly Gaussian 
random variables with mean 0 and covariance structure as follows. Diagonal entries $\Re 
J_{j,j}, \Im J_{j,j}$ have the variances
\begin{align}
 \s^2_{D,\Re}:=\frac{1+\t}{2N(1+2\g K(1+\t))},\quad \s^2_{D,\Im}:=\frac{1-\t}{2N(1+2\g K(1-\t))},\label{variance_diag}
\end{align}
respectively,
off-diagonal entries $\Re 
J_{j,k}, \Im J_{j,k}, j\not=k$ the variance
\begin{align}
 \s^2_O:=\frac{1+2\g K(1-\t^2)}{2N(1+4\g K+4\g^2K^2(1-\t^2))}\label{variance_off}
\end{align}
and the covariances are 0 except for ($j\not=k$)
\begin{align}
\rho:=\textup{Cov}(\Re 
J_{j,k},\Re 
J_{k,j})=\frac{\t}{2N(1+4\g K+4\g^2K^2(1-\t^2))}\label{cov_real}
\end{align}
and 
\begin{align*}
 \textup{Cov}(\Im 
J_{j,k},\Im 
J_{k,j})=-\frac{\t}{2N(1+4\g K+4\g^2K^2(1-\t^2))}=-\rho.%\label{cov_imag}
\end{align*}

The matrix $J$ can be associated with a vector $\underline{J}\in\R^{2N^2}$ that has a multivariate normal distribution with mean 
$0$ and covariance matrix $\Sigma$ that is of block diagonal form. Each block consists either of $\s_{D,\Re}^2,\ \s_{D,\Im}^2$, or $2\times 2$ 
matrices with $\s_{O}^2$ on the diagonal and $\pm\rho$ off the diagonal. $\underline{J}$ has the same distribution as 
$\underline{U}\underline{Y}$ where $\underline{U}$ is a $2N^2\times 2N^2$ unitary matrix (diagonalizing $\Sigma$) and 
$\underline{Y}$ is Gaussian with 
independent components with mean 0. The variances of the elements of $\underline{Y}$ are $\s_{D,\Re}^2$ or $\s_{D,\Im}^2$ for $N$ 
components each, and $\l_+^2$ or $\l_-^2$ for $N(N-1)$ 
components each, where 
\begin{align*}
 \l_\pm^2:=\s_O^2\pm\rho%\label{variances_independent}
\end{align*}
are the eigenvalues of those blocks of $\Sigma$ that contain off-diagonal entries. As 
$\Tr(JJ^*)=\|\underline{J}\|_2^2\overset{d}{=}\|\underline{Y}\|_2^2$, where $\overset{d}{=}$ means equality of distributions, 
we find that 
\begin{align}
 \Tr(JJ^*)\overset{d}{=}\l_+^2Z_1+\l_-^2Z_2+\s_{D,\Re}^2Z_3+\s_{D,\Im}^2Z_4,\label{chi}
\end{align}
where $Z_1,Z_2,Z_3,Z_4$ are independent $\chi^2$ distributed random variables with $N(N-1)$, \newline$N(N-1),N$ and $N$ degrees of 
freedom, 
respectively.

Let us now choose $K=K(\t,\g,K_p)$ such that 
\begin{align*}
	K_p+K=2N\s_O^2,
\end{align*}
 in other words
 \begin{align}
  K_p+K=\frac{1+2\g K(1-\t^2)}{1+4\g K+4\g^2K^2(1-\t^2)}.\label{condition2}
 \end{align}
Note that $K$ does not depend on $N$. It is not hard to see that the r.h.s.~has poles at $K=-(2\g(1\pm\t))^{-1}$ and is strictly 
decreasing in $K$ for $K>\max\{-(2\g(1-\t))^{-1},-(2\g(1+\t))^{-1}\}$. As 
 the l.h.s.~of \eqref{condition2} is a strictly increasing continuous function on $\R$, there is precisely one 
 $K>\max\{-(2\g(1-\t))^{-1},-(2\g(1+\t))^{-1}\}$ satisfying \eqref{condition2}. It is the rightmost (real) solution of the cubic equation
 \begin{align}
  4\g^2(1-\t^2) K^3+4\g(1+\g(1-\t^2)K_p) K^2+(1+4\g K_p-2\g(1-\t^2))K+K_p-1=0.\label{cubic}
 \end{align}
 It can in principle be computed explicitly but its exact form is not important here.
 
 As $\t\to1$, \eqref{condition2} becomes
 \begin{align*}
  K_p+K=\frac{1}{1+4\g K},%\label{condition3}
 \end{align*}
 which has, analogously to \eqref{condition2}, only one solution $\bar{K}:=\lim_{\t\to1}K(\t)$ that satisfies $\bar{K}>-\frac1{4\g}$. It is given by
 \begin{align*}
  \bar{K}=-\frac{K_p}{2}-\frac1{8\g}+\frac1{8\g}\sqrt{16\g^2 K_p^2-8\g K_p+16\g+1}.
 \end{align*}
 Furthermore, by applying the implicit function theorem to the function
 \begin{align*}
  F(\t,K):=4\g^2(1-\t^2) K^3+4\g(1+\g(1-\t^2)K_p) K^2+(1+4\g K_p-2\g(1-\t^2))K+K_p-1,
 \end{align*}
 we find that the derivative $K'(\t)$ is bounded in a neighborhood of $\t=1$ and thus \eqref{K_derivative} follows. This proves b).

 Let us turn to proving a). We have by \eqref{chi} and \eqref{condition2}
\begin{align}\label{sum_rv}
	\Tr(JJ^*)-N(K_p+K)\overset{d}{=}\bar{\l}_+^2\sqrt{2}\frac{Z_1-N(N-1)}{\sqrt{2N^2}}+\bar{\l}_-^2\sqrt{2}\frac{Z_1-N(N-1)}{\sqrt{2N^2}
	}+\s_{D,\Re}^2Z_3+\s_{D,\Im}^2Z_4,
\end{align}
where $\bar{\l}_\pm^2:=2N\l_\pm^2$ do not depend on $N$. Now $(Z_j-N(N-1))/\sqrt{2N^2}, j=1,2$ converge weakly towards  standard 
normals as $N\to\infty$ and $\s_{D,\Re}^2Z_3,\s_{D,\Im}^2Z_4$ converge weakly to constants. Since $Z_i$, $i=1,2,3,4$ are independent, we have weak convergence of \eqref{sum_rv} to a Gaussian distribution and hence
\begin{align*}
	\lim_{N\to\infty} \E_{a,b}\exp\lr{-\g 
		(\Tr \tilde{J}\tilde{J}^*-N(K_p+K))^2}\rr=\exp\lr{-\frac{C\g}{1+c\g}}\rr\lb 1+c\g\rb^{-1/2}
\end{align*}
for some $C,c>0$, i.e.~convergence to the moment-generating function of a non-central chi-squared distribution. This proves \eqref{bounds}.

If $\lv\t-1\rv=\O(1/N)$, \eqref{sum_rv} holds with $\bar{K}$ replacing $K$ and adding $N(K-\bar K)=\O(1)$ to the r.h.s.~of \eqref{sum_rv}. This shows c).

Similar reasoning as above shows with \eqref{condition2} that $\g K$ solves for $\g\to\infty$ the equation
\begin{align*}
	 4(1-\t^2)K_p(\g K)^2+(4K_p-2(1-\t^2))(\g K)+K_p-1=0,
\end{align*}
from which the first part of d) can be checked. The second part follows immediately.
\end{proof}

From now on, let $K$ always denote the quantity from Lemma \ref{Lemma_K}.
The linearization uses the simple identity
\begin{align}
 \exp\lr{-\g X^2}\rr=\frac{1}{\sqrt{4\pi\g}}\int_\R \exp\lr{iXt}\rr \exp\lr{-\frac{t^2}{4\g}}\rr dt,\label{linearization}
\end{align}
valid for any real $X$.
In the physics literature, this is known as the Hubbard-Stratonovich transform (in its simplest form). With \eqref{linearization}, we may rewrite 
$P_{N,\tr}$ as 
\begin{align*}
 P_{N,\tr}(J)=\frac{1}{\sqrt{4\pi\g}}\int_\R&\frac{\exp\lr{-iN(K_p+K)t}\rr}{Z_{N,\tr} \exp\lr{-\g N^2(K^2+2KK_p)}\rr}\\
&\times \exp\lr{\lb-a+it\rb\Tr JJ^*+\frac{b}{2}\Tr \lb J^2+{J^*}^2\rb}\rr \exp\lr{-\frac{t^2}{4\g}}\rr dt.
\end{align*}
Setting
\begin{align*}
	a(t):=a-it,%\label{def_a}
\end{align*}
we have (extending the definition \eqref{K_density} to the complex $a(t)$)
\begin{align*}
 \frac{Z_{a(t),b}}{Z_{a,b}}=\E_{a,b}\exp\lr{it\Tr \tilde{J}\tilde{J}^*}\rr,
\end{align*}
and thus by \eqref{chi}, $Z_{a(t),b}/Z_{a,b}$ is the product of characteristic functions of $\chi^2$ distributed random 
variables. Hence the 
function
\begin{align*}
 t\mapsto \E_{a,b}\exp\lr{it\Tr \tilde{J}\tilde{J}^*}\rr%\label{partition_function}
\end{align*}
has no zeros on the real line and the ``linearized ensemble''
\begin{align*}
  &P_{a(t),b}(J)=\frac{1}{Z_{a(t),b}}\exp\lr{\lb-N\lb\frac {1}{1-\t^2}+2\g 
K\rb+it\rb\Tr JJ^*+\frac{\t N}{2(1-\t^2)}\Tr \lb J^2+{J^*}^2\rb}\rr
\end{align*}
is well-defined.
The term ``ensemble'' here is only a convenient naming, in general $P_{a(t),b}$ is complex-valued. Summarizing, we arrive at
\begin{align*}
	P_{N,\tr}(J)=\frac1{\sqrt{4\pi\g}}\int_\R\frac{Z_{a(t),b}\exp\lr{-iN(K_p+K)t}\rr}{Z_{N,\tr} \exp\lr{-\g N^2(K^2+2KK_p)}\rr}P_{a(t),b}(J)\exp\lr{-\frac{t^2}{4\g}}\rr dt,
\end{align*}
which in turn can be rewritten (after multiplying and dividing by $Z_{a,b}$) as
\begin{align}
	P_{N,\tr}(J)=\frac1{\sqrt{4\pi\g}}\int_\R \frac{\E_{a,b}\exp\lr{it 
		(\Tr \tilde{J}\tilde{J}^*-N(K_p+K))}\rr}{\E_{a,b}\exp\lr{-\g 
		(\Tr \tilde{J}\tilde{J}^*-N(K_p+K))^2}\rr}P_{a(t),b}(J)\exp\lr{-\frac{t^2}{4\g}}\rr dt.\label{linearization_full}
\end{align}

One advantage of the 
linearization is that the joint distribution of eigenvalues (also in the sense of a complex-valued measure) of $P_{a(t),b}$ can 
be given explicitly as
\begin{align}
   &P_{a(t),b}(z):=\frac{1}{Z_{a(t),b}^{\textup{EV}}}\prod_{j<l}\lv z_j-z_l\rv^2\exp\lr{-a(t)\sum_{j=1}^N \lv z_j\rv^2+\frac{b}{2}(\sum_{j=1}^N 
z_j^2+\overline{z_j}^2)}\rr,\label{EV_density_epsilon}
\end{align}
where we abused notation by using the same symbol for the matrix and the eigenvalue distribution. The superscript EV in the 
normalization constant indicates that $Z_{a(t),b}$ and $Z_{a(t),b}^{\textup{EV}}$ differ.  The density 
\eqref{EV_density_epsilon} will be analyzed in the next section. By continuity it is clear that the joint distribution of eigenvalues of $P_{N,\tr}$ also has a continuous density. Hence we can speak of its $k$-th correlation function $\rho_{N,\tr}^{k}$, which we define as in \eqref{def_correlation_functions}.

In terms of correlation functions, we get from \eqref{linearization_full}
\begin{align}
	\rho_{N,\tr}^{k}(z)=\frac1{\sqrt{4\pi\g}}\int_\R \frac{\E_{a,b}\exp\lr{it 
			(\Tr \tilde{J}\tilde{J}^*-N(K_p+K))}\rr}{\E_{a,b}\exp\lr{-\g 
			(\Tr \tilde{J}\tilde{J}^*-N(K_p+K))^2}\rr}\rho_{a(t),b}^{k}(z)\exp\lr{-\frac{t^2}{4\g}}\rr dt,\label{identity_corr1}
\end{align}
where $\rho_{a(t),b}^{k}$ denotes the $k$-th correlation function of $P_{a(t),b}$ and $z$ is an abbreviation for 
$z_1,\dots,z_k$. Moreover,
\begin{align}
		&\check{\rho}_{N,\tr}^{k}(\check z)-\det(K_{\weak/\strong}(z_j,z_l))_{j,l\leq k}\notag\\
		&=\frac1{\sqrt{4\pi\g}}\int_\R \frac{\E_{a,b}e^{it 
				(\Tr \tilde{J}\tilde{J}^*-N(K_p+K))}}{\E_{a,b}e^{-\g 
				(\Tr 
\tilde{J}\tilde{J}^*-N(K_p+K))^2}}\lb\check{\rho}_{a(t),b}^{k}(\check{z})-\det(K_{\weak/\strong}(z_j,z_l))_{j,l\leq k}\rb 
e^{-\frac{t^2}{4\g}}dt,\label{identity_corr_functions}
\end{align}
where $\check{\rho}_{N,\tr}^{k}(\check z)$ and $\check{\rho}_{a(t),b}^{k}(\check{z})$ denote the appropriately rescaled correlation 
functions of the rescaled variables, e.g.~in the situation of Theorem \ref{theorem_strong} b) 
\begin{align}
 \check\rho^k_{N,\tr}=\frac{1}{(CN)^k}\rho_{N,\tr}^k,\quad \check z_j=Z+\frac{z_j}{\sqrt{CN}},\quad j=1,\dots,k.\label{rescaling}
\end{align}

\begin{remark}\label{non-determinantal}
From \eqref{identity_corr1}, we see the non-determinantality of the eigenvalue distribution of $P_{N,\tr}$. Its correlation functions are not determinants themselves, but rather averages of determinants.
\end{remark}

%%%%%%%%%%%%%%%%%%%%%%%%%%%%%%%%%%%%%%%%%%%%%%%%%%%%%%%%%%%%%%%%%
\section{Asymptotics for the Linearized Ensemble}\label{section_soft}
In this section, we will employ orthogonal polynomials and asymptotic analysis to obtain the asymptotic behavior of the linearized correlation functions $\rho_{a(t),b}^{k}$, as $N\to\infty$. For $\t\in(-1,1),\t\not=0$, the orthogonal polynomials are Hermite polynomials, whereas for $\t=0$ the orthogonal polynomials are simple monomials. We will first concentrate on the much more involved case $\t\not=0$, the case $\t=0$ will be dealt with at the end of the proof of Proposition \ref{strong_linearized} below.

Let $(H_k)_{k\in\mathbb{N}}$ denote the sequence of Hermite polynomials, that is (cf. \cite[22.10.0]{AbramowitzStegun})
\begin{align}
 H_k(z):=\frac{k!}{2\pi i}\oint \exp\lr{-{t^2}+2zt}\rr t^{-(k+1)}dt,\label{Hermite_polynomial}
\end{align}
where the contour encircles the origin. It is well-known that these polynomials form an orthogonal sequence in $L^2(\R)$ w.r.t.~the 
weight $\exp(-t^2)$, i.e. for $k\not=l$
\begin{align*}
 \int H_k(t)H_l(t)e^{-t^2}dt=0.
\end{align*}
Our analysis depends crucially on the fact that the Hermite polynomials are also orthogonal on the complex plane w.r.t.~certain 
Gaussian measures, a fact first noticed in \cite{EM90,DiFrancescoetal94}. This can be translated to complex weights easily.
Recall first
\begin{align*}
	a(t)=N\lb\frac {1}{1-\t^2}+2\g 
	K\rb-it,\qquad b=\frac{\t N}{1-\t^2},
\end{align*}
and $a(0)=a$. Since $b$ will remain unchanged, we will omit it in newly defined quantities to ease the notation.
\begin{lemma}\label{lemma_weight}
 Consider the function
\begin{align*}
 W_{a(t)}(z):=\exp\lr{-a(t)\lv z\rv^2+\frac b2(z^2+\bar{z}^2)}\rr.
\end{align*}
 Then
\begin{align}
 \int_\C H_l(c_{a(t)}z)H_k(c_{a(t)}\bar{z})W_{a(t)}(z)dz=\delta_{lk}\frac{k!\pi(2a(t))^k}{\sqrt{a(t)^2-b^2}\,b^k},\label{Hermite}
\end{align}
where \begin{align*}
       c_{a(t)}:=\sqrt{\frac{a(t)^2-b^2}{2b}}
      \end{align*}
and $\sqrt{\cdot}$ denotes the principal branch.
\end{lemma}

\begin{proof} 
Using the integral representation \eqref{Hermite_polynomial} and the residue theorem one can easily verify \eqref{Hermite} which we leave to the reader.
\end{proof}

Define $(p_k)_{k\in\mathbb N}$ as the sequence of orthonormal polynomials $p_k(z):=C_{a(t),k}H_k(c_{a(t)}z)$ with 
 $C_{a(t),k}:=\lbb\frac{k!\pi(2a(t))^k}{\sqrt{a(t)^2-b^2}b^k}\rbb^{-1/2}$ such that 
\begin{align*}
 \int_\C p_l(z)p_k(\overline{z})W_{a(t)}(z)dz=\d_{lk}.
\end{align*}
Now, using standard arguments, it is seen that the ensemble $P_{a(t),b}$ is determinantal, i.e. its correlation functions 
$\rho_{a(t),b}^k$, $k=1,2,\dots$,
\begin{align*}
 \rho_{a(t),b}^k(z_1,\dots,z_k)=\frac{N!}{(N-k)!}\int_{\C^{N-k}} P_{a(t),b}(z)dz_{k+1}\dots dz_N
\end{align*}
fulfill
\begin{align*}
 \rho_{a(t),b}^k(z_1,\dots,z_k)=\det\lb K_{a(t)}(z_i,z_j)\rb_{1\leq i,j\leq k}
\end{align*}
with the kernel
\begin{align}
 K_{a(t)}(z_1,z_2):=\sum_{j=0}^{N-1}p_j(z_1)p_j(\overline{z_2})\sqrt{W_{a(t)}(z_1)}\sqrt{W_{a(t)}(\overline{z_2})}.\label{def_kernel}
\end{align}
The analysis of correlation functions of $P_{a(t),b}$ thus boils down to an analysis of the kernel $K_{a(t)}$. In the limit of weak 
non-Hermiticity  we will prove
\begin{prop}\label{prop_kernel_linearized}
Define $C_{\bar{K}}:=1+4\g \bar{K}$ with $\bar{K}$ from Lemma \ref{Lemma_K}. Let $X\in (-\frac{2}{\sqrt{C_{\bar{K}}}},\frac2{\sqrt{C_{\bar{K}}}})$. As $N\to\infty$, we have with  $\t=\t_N:=1-\dfrac{\tilde{\a}^2}{2C_{\bar{K}}^2N}$, $\tilde{\a}>0$,
\begin{align*}
&\frac{1}{C_{\bar{K}}^2N^2}K_{a(t)}\lb 
X+\frac{x_1+iy_1}{C_{\bar{K}}N},X+\frac{x_2+iy_2}{C_{\bar{K}}N}\rb=\frac{1}{\pi}\exp\Big[-\frac{
y_1^2+y_2^2}{ \tilde{\a}^2 } +iX\frac{(y_1-y_2)}2\Big]\\
&\times \frac1{\sqrt{2\pi}\tilde{\a}}
\int_{-\frac12\sqrt{\frac{4}{C_{\bar{K}}}-X^2}}^{\frac12\sqrt{\frac{4}{C_{\bar{K}}}-X^2}}\exp\Big[-\frac{\tilde{\a}^2u^2}{2}
+iu(x_1-x_2)-u(y_1+y_2)\Big]
du 
+\O\lb \frac{\log N}{\sqrt{N}}\rb.
\end{align*}
The $\O$ term is uniform for $X\in (-\frac{2}{\sqrt{C_{\bar{K}}}}+\d,\frac2{\sqrt{C_{\bar{K}}}}-\d)$ for any $\d>0$ fixed, any $x_j,y_j, j=1,2$ chosen from an arbitrary compact subset of $\R$ and $\lv t\rv=\O(\sqrt{\log N})$.
\end{prop}
\begin{remark}\label{remark_uniform}
	\begin{enumerate}
\item The reader will note that $\tilde{\a}$ in Proposition \ref{prop_kernel_linearized} differs from $\a$ in Theorem \ref{theorem_weak} by a factor $\nu(X)^{-1}$. The scaling of the local variables $z_j$ differs also in the same way. This difference is  purely due to notational convenience.
\item The $\O$ term is also uniform in $\g \bar{K}$ in the following sense. Treating $\g \bar{K}$ as an independent variable, Proposition \ref{prop_kernel_linearized} holds uniformly for $\g\bar{K}$ in any compact subset of $(-1/4,\infty)$. This is needed in Section \ref{Sec_proof_FT} when dealing with the fixed trace ensembles.
\end{enumerate}
\end{remark}
The limit of strong non-Hermiticity corresponds to a fixed $\t\in(-1,1)$. 
\begin{prop}\label{strong_linearized}
 Let $\t\in(-1,1)$ be fixed and $k$ be a nonnegative integer. As $N\to\infty$, we have with $C:=\frac{1}{1-\t^2}+2\g K$
\begin{align*}
 &\frac{1}{(CN)^k}\rho^k_{a(t),b}\lb Z+\frac{z_1}{\sqrt{CN}},\dots,Z+\frac{z_k}{\sqrt{CN}}\rb=\det(K_{\strong}({z_j,z_l})_{j,l\leq k}+\O\lb\frac{1}{\sqrt{N}}\rb,
\end{align*}
where $K_{\strong}$ has been defined in \eqref{kernel_strong}.
Here $Z$ is chosen from the interior of the elliptic set
\begin{align*}
E:=\left\{Z\in\C\,:\,\frac{1-\t+2\g K(1-\t^2)}{1+\t+2\g K(1-\t^2)}\Re Z^2+\frac{1+\t+2\g K(1-\t^2)}{1-\t+2\g K(1-\t^2)}\Im Z^2\leq\frac{1}{C}\right\}.
\end{align*}
The $\O$ term is uniform for $Z$ from any compact subset of $E^\circ$, $z_1,\dots,z_k$ from compacts of $\C$ and $\lv t\rv=\O(\sqrt{\log N})$.
\end{prop}
\begin{remark}
	The $\O$ term is uniform for $\g K$ from compact subsets of $(-(2(1+\t))^{-1},\infty)$ (cf.~Remark \ref{remark_uniform}).
\end{remark}

A technical ingredient in 
the proof of Propositions \ref{prop_kernel_linearized} and \ref{strong_linearized} is stated in the following lemma for the normalized upper incomplete gamma 
function
\begin{align*}
 Q(w,z):=\frac{\Gamma(w,z)}{\Gamma(w)},\quad \Gamma(w,z):=\int_z^\infty t^{w-1}e^{-t}dt.
\end{align*}
In these definitions $w$ and $z$ are real and positive but $\Gamma(w,z)$ and $Q(w,z)$ can in fact be continued to analytic 
functions in the complex plane, provided $w>0$. We will use the same symbols for these continued functions. 
\begin{lemma}[\cite{Temme}]\label{lemma_gamma}\noindent
Denote $\eta:=\sqrt{2(z-1-\log z)}$, where we 
choose the branch of the square root such that it has  the same sign as $z-1$ for real positive $z$, and by continuity elsewhere. 
Then, as $w\to\infty$,
\begin{align}
 Q(w,wz)=\frac{1}{2}\erfc\lb\eta\sqrt{(w/2)}\rb+\O\lb \frac{e^{-\frac w2\eta^2}}{\sqrt{w}}\rb,\label{Temme1}
\end{align}
where $\erfc$ denotes the complementary error function
\begin{align*}
 \erfc(z):=\frac{2}{\sqrt{\pi}}\int_z^\infty e^{-t^2}dt,
\end{align*}
and the $\O$ term is uniform 
in the domain $\lv\arg(z)\rv\leq 3\pi/2-\d$ with arbitrary $\d>0$, i.e.~for $z=re^{i\phi}$ with 
$-2\pi+\d\leq\phi\leq2\pi-\d, r>0$.
\end{lemma}

\begin{proof}
The lemma is a special case of an asymptotic expansion derived in \cite{Temme}, where it is shown that
\begin{align*}
 Q(w,wz)=\frac{1}{2}\erfc\lb\eta\sqrt{(w/2)}\rb+R_w(\eta)
\end{align*}
and the remainder $R_w(\eta)$ in 
\eqref{Temme1}  admits an asymptotic expansion in negative powers of $w$ as $w\to\infty$,
\begin{align*}
 R_w(\eta)\sim(2\pi w)^{-1/2}e^{-\frac12 w\eta^2}\sum_{k=0}^\infty c_k(\eta)w^{-k}.
\end{align*}
The expansion is uniform for $\eta$ in the domain  
$\lv\arg(z)\rv\leq 3\pi/2-\d$ with $\d$ being an arbitrarily small positive constant (see also \cite{DPC}).
\end{proof}

Before beginning with the proof of Proposition \ref{prop_kernel_linearized}, we state the relevant asymptotics of the complementary 
error function for later use.
From \cite[6.5.32]{AbramowitzStegun} we have 
\begin{align}
 \erfc(z) =\frac{e^{-z^2}}{\sqrt{\pi}z}\lb 1+\O\lb{z^{-2}}\rb \rb,\label{erfc_asymptotics1}
\end{align}
as $z\to\infty$ in $\lv\arg(z)\rv<3\pi/4$. Morover, we have by the relation $\erfc(-z)=2-\erfc(z)$
\begin{align}
 \erfc(-z) =2-\frac{e^{-z^2}}{\sqrt{\pi}z}\lb 1+\O\lb{z^{-2}}\rb \rb,\label{erfc_asymptotics2}\quad z\to\infty,
\end{align}
again with $\lv\arg(z)\rv<3\pi/4$.

\begin{proof}[Proof of Proposition \ref{prop_kernel_linearized}]
We follow the path outlined in \cite{FKS} for the simpler case $t=0, K=0$ and start with an alternative representation of the 
Hermite polynomials (cf. \cite[22.10.15]{AbramowitzStegun} for the  sign $(-2i)^k$, for the other sign use the parity 
$H_k(-z)=(-1)^kH_k(z)$),
\begin{align*}
 H_k(z)=\frac{(\pm 2i)^k}{\sqrt{\pi}}\exp\lr{{z^2}}\rr\int_\R r^k\exp\lr{-{r^2}\mp 2izr}\rr dr,
\end{align*}
which allows to rewrite the kernel as
\begin{align}
 K_{a(t)}(z_1,z_2)&=\frac{\sqrt{a(t)^2-b^2}}{\pi^2}\exp\lr{c_{a(t)}^2}(z_1^2+\overline{z_2}^2)-\frac {a(t)}2(\lv z_1\rv^2+\lv 
z_2\rv^2)+\frac b4(z_1^2+z_2^2+\overline{z_1}^2+\overline{z_2}^2)\rr\notag\\
&\times\sum_{j=0}^{N-1}\lb \frac {2b}{a(t)}\rb^j\frac1{j!}\int_\R\int_\R 
(rs)^j\exp\lr-(r^2+s^2)+2ic_{a(t)}(rz_1-s\overline{z_2})\rr drds.\label{kernel1}
\end{align}
It will be useful lateron to derive a general formula independent of the limit regime. To this end we decouple the integrals via the substitution
\begin{align*}
 \Phi(r,s):=(u,v):=\lb\frac{r+s}{c_{a(t)}},\frac{r-s}{c_{a(t)}}\rb, \lv\det D\Phi(r,s)\rv=\frac2{\lv c_{a(t)}\rv^2}
\end{align*}
and arrive after algebraic manipulations at
 \begin{align}
 &K_{a(t)}(z_1,z_2)=\frac{\sqrt{a(t)^2-b^2}\lv c_{a(t)}\rv^2}{2\pi^2} \exp\lr-\frac {a(t)}2(\lv z_1\rv^2+\lv 
 z_2\rv^2)+\frac {ib}2\lb\Im(z_2^2)-\Im(z_1^2)\rb+a(t)z_1\overline{z_2}\rr\notag\\
 &\times\int_{\R/c_{a(t)}}\exp\lr-\frac{(a(t)+b)(a(t)-b)^2}{4a(t)b} \lb u-\frac{i(z_1-\overline{z_2})}{1-b/a(t)}\rb^2\rr\notag\\
 &\times\int_{\R/c_{a(t)}}\exp\lr-\frac{(a(t)+b)^2(a(t)-b)}{4a(t)b} \lb v-\frac{i(z_1+\overline{z_2})}{1+b/a(t)}\rb^2\rr Q\lb N,\frac{c_{a(t)}^2b}{2a(t)}(u^2-v^2)\rb dvdu.\label{gamma1}
 \end{align}
Here we used the definition of $c_{a(t)}$ in Lemma \ref{lemma_weight} and the well-known fact that for a positive integer $w$ one has
\begin{align*}
 Q(w,z)=e^{-z}\sum_{j=0}^{w-1}\frac{z^j}{j!}.%\label{Q_sum}
\end{align*}

We will consider the two integrals iteratively and start with the $v$-integral. Setting 
\begin{align}
C_{z_1,z_2}:=\frac{i(z_1+\overline{z_2})}{1+b/a(t)}, \label{def_C_z_1}
\end{align}
it can be written as
\begin{align}
 &\int_{\R/c_{a(t)}+C_{z_1,z_2}}\exp\lr-\frac{(a(t)+b)^2(a(t)-b)}{4a(t)b} v^2\rr \notag\\
&\times Q\lb N,\frac{c_{a(t)}^2b}{2a(t)}(u^2-v^2)-i\frac{c_{a(t)}^2b(z_1+\overline{z_2})}{a(t)+b}v+\frac{c_{a(t)}^2b(z_1+\overline{z_2})^2}{2a(t)(1+b/a(t))^2}\rb dv.\label{v-integral}
\end{align}
Our next aim is to truncate the integral over the infinite line $\R/c_{a(t)}+C_{z_1,z_2}$ to one over a small compact set.
With 
\begin{align*}
 z_j=X+\O\lb\frac1N\rb,\quad j=1,2 
\end{align*}
and $X\in\R$,
 it is easy to check that, as $N \to\infty$, $\R/c_{a(t)}+C_{z_1,z_2}\to\R+iX$. Moreover, we have the asymptotics
 \begin{align}
 	&\frac{(a(t)+b)^2(a(t)-b)}{4a(t)b}=\frac{NC_{\bar{K}}}{2}+\O\lb\frac{1+\lv t\rv}{N}\rb,\label{decay1}\\
 	&\frac{c_{a(t)}^2b}{2a(t)}=\frac{NC_{\bar{K}}}{4}+\O(1),\label{decay2}\\
 	&\frac{c_{a(t)}^2b}{a(t)+b}=\frac{NC_{\bar{K}}}{4}+\O(1),\label{decay3}\\
 	&\frac{c_{a(t)}^2b}{2a(t)(1+b/a(t))^2}=\frac{NC_{\bar K}}{16}+\O(1),\label{decay4}
 \end{align}
 where we recall $C_{\bar{K}}=1+4\g \bar{K}$. All $\O$ terms might be complex-valued and are uniform in $t$. We also note in passing that for these asymptotics, \eqref{K_derivative} has been used.

  \eqref{decay1} is relevant for the decay of the Gaussian term of the integral, (\ref{decay2},\ref{decay3},\ref{decay4}) are needed to control the possible growth of the incomplete Gamma function $Q(N,\cdot)$.  In view of Lemma \ref{lemma_gamma}, its first order asymptotics are given by $\erfc(\sqrt{N/2}\eta)/2$. If $\sqrt{N/2}\lv\eta\rv$ is bounded in $N$, the analyticity of $\erfc$ yields boundedness of the $Q$-term. If $\sqrt{N/2}\lv\eta\rv\to\infty$, \eqref{erfc_asymptotics1} or \eqref{erfc_asymptotics2} are applicable and we get the term $\exp(-\eta^2N/2)=\exp(-N(z-1-\log z))$ with 
 \begin{align*}
 	z=\frac{C_{\bar{K}}}{4}(u^2+X^2-v^2-2ivX)+\O(\lv t\rv/N).
 \end{align*}
 From \eqref{decay1} we see for $u$ fixed that the decay of $\exp\lr-\frac{(a(t)+b)^2(a(t)-b)}{4a(t)b} v^2\rr$ in $v$ is  faster than the possible growth of the $Q$-term and thus we can up to an error of order $\O(1/N)$ truncate the integral over $\R/c_{a(t)}+C_{z_1,z_2}$ to $[-CN^{-1/2+\e},CN^{-1/2+\e}]+C_{z_1,z_2}$ for some $C>0$ and any $\e>0$.
 Since $\R/c_{a(t)}$ is close to the real line for $N$ large enough, we can further assume $\Re(u^2)>0$ or $\Re z>-c$ for some small $c>0$. For such $z$ we have $\Re(z-1-\log z)>0$ or equal $+\infty$ at $0$ and hence we get $\lv e^{-\eta^2N/2}\rv\leq1$. This shows that the constant $C$ in $\lv v\rv\leq CN^{-1/2+\e}$ can be chosen independent of $u$, i.e.~our estimates are uniform in $u$.
 
 Next we will see that
 \begin{align}
 	Q\lb N,Nz\rb=1_{\{u^2+X^2\leq 
 		4/C_{\bar{K}}\}}+\O\lb\frac{\log N}{\sqrt{N}}\rb\label{eqn_indicator}
 \end{align}
 with an $\O$ term that is uniform in $u$ and $v$. Since $\lv v^2\rv=\O(1/N)$, the uniformity in $v$ is trivial. If $u$ is such that $\Re (u^2)+X^2\geq 4/C_{\bar K}+\sqrt{(\log N)/N}$, then $Q(N,Nz)=\O(1/N)$ by \eqref{erfc_asymptotics1}. Similarly, if $\Re (u^2)+X^2\leq 4/C_{\bar K}-\sqrt{(\log N)/N}$, then $Q(N,Nz)=1+\O(1/N)$ by \eqref{erfc_asymptotics2}. The area of $u$'s on $\R/c_{a(t)}$ such that $\lv \Re (u^2)+X^2-4/C_{\bar K}\rv\leq\sqrt{(\log N)/N}$, has a volume of order $\O((\log N)/N)$. Since $\erfc(\eta\sqrt{N/2})$ is bounded on this area, we arrive at \eqref{eqn_indicator}.

The remaining $v$-integral is
\begin{align*}
	&\int_{[-CN^{-1/2+\e},CN^{-1/2+\e}]+C_{z_1,z_2}}\exp\lr-\frac{NC_{\bar K}}{2} v^2\rr dv\sim \sqrt{\frac{2\pi}{NC_{\bar{K}}}},
\end{align*}
which can be seen by using analyticity and standard arguments to shift the contour to the real line.

To finish the proof, we choose in \eqref{gamma1}
\begin{align*}
 z_j=X+\frac{x_j+iy_j}{C_{\bar{K}}N},\quad j=1,2
\end{align*}
with $x_j,y_j\in\R,j=1,2$
and note the asymptotics
\begin{align*}
	&\frac{(a(t)+b)(a(t)-b)^2}{4a(t)b}=\frac{\tilde{\a}^2}{8}+\O\lb\frac{1+\lv t\rv}{N}\rb,\\
	&\sqrt{a(t)^2-b^2}\lv c_{a(t)}\rv^2=\frac{N^{5/2}C_{\bar{K}}^{5/2}}{2\tilde{\a}}+\O(N^{3/2})
	,\\
	&1-\frac b{a(t)}=\frac{\tilde{\a}^2}{2C_{\bar{K}}N}+\O\lb\frac t{N^2}\rb,\\
	&a(t)\sim b=\frac{N^2C_{\bar K}^2}{\tilde{\a}^2}+\O(N).
\end{align*}
 The proposition now follows by the change of variables $u\mapsto u/2$.
\end{proof}

\begin{proof}[Proof of Proposition \ref{strong_linearized}]
For the limit of strong non-Hermiticity, we choose
\begin{align}
 z_j=X+iY+\frac {x_j}{\sqrt{N}}+i\frac {y_j}{\sqrt{N}},\quad j=1,2\label{z_strong}
\end{align}
with $X,Y,x_j,y_j\in \R$. We deal with the more complicated case $\t\not=0$ first.
Using \eqref{z_strong}, \eqref{gamma1} reads
\begin{align}
 &K_{a(t)}(z_1,z_2)=\frac{\sqrt{a(t)^2-b^2} 
\lv c_{a(t)}\rv^2}{2\pi^2}\exp\Big[-\frac{a(t)(x_1^2+x_2^2)}{2N}-\frac{a(t)(y_1^2+y_2^2)}{2N}+\frac{i(a(t)-b)X(y_1-y_2) } { \sqrt{N}
} \notag\\
&-\frac{i(a(t)+b)Y(x_1-x_2)}{\sqrt{N}}
-ib\frac{(x_1y_1-x_2y_2)}{N}+\frac{a(t)(x_1x_2+y_1y_2-ix_1y_2+ix_2y_1)}{N}\Big]\notag\\
&\times\int_{\R/c_{a(t)}}
\exp\Big[-\frac{(a(t)-b)^2(a(t)+b)}{4a(t)b}\lb 
u+\frac{a(t)}{a(t)-b}\lb2Y-\frac{i(x_1-x_2)}{\sqrt{N}}+\frac{y_1+y_2}{\sqrt{N}}\rb\rb^2\Big]\notag\\
&\times\int_{\R/c_{a(t)}} 
\exp\Big[-\frac{(a(t)+b)^2(a(t)-b)}{4a(t)b}\lb v-\frac{a(t)}{a(t)+b}\lb 
2iX+\frac{i(x_1+x_2)}{\sqrt{N}}-\frac{y_1-y_2}{\sqrt{N}}\rb\rb^2\Big]\notag\\
&\times Q\lb N,\frac{c_{a(t)}^2b}{2a(t)}(u^2-v^2)\rb 
dvdu.\label{gamma2}
\end{align}
 We can now proceed analogously to the proof of Proposition \ref{prop_kernel_linearized}. For the $v$-integral we get 
\begin{align*}
&\int_{\R/c_{a(t)}}
\exp\Big[-\frac{(a(t)+b)^2(a(t)-b)}{4a(t)b}\lb v-\frac{a(t)}{a(t)+b}\lb 
2iX+\frac{i(x_1+x_2)}{\sqrt{N}}-\frac{y_1-y_2}{\sqrt{N}}\rb\rb^2\Big]\notag\\
&\times \frac{1}{\sqrt{2\pi}}\sqrt{\frac{(a(t)+b)^2(a(t)-b)}{2a(t)b}}Q\lb N,\frac{c_{a(t)}^2b}{2a(t)}(u^2-v^2)\rb 
dv\\
&= 1_{\big\{\frac{c_{a}^2b}{2aN}u^2+\frac{a(a-b)}{(a+b)N}X^
2\leq 1\big\}}+ \O(1/\sqrt{N}),
\end{align*}
where $a=a(0)$. Here, we used that the weaker error term in \eqref{eqn_indicator} can be improved (even to exponentially fast decaying terms) if $u$ is such that $\frac{c_{a}^2b}{2aN}u^2+\frac{a(a-b)}{(a+b)N}X^
2\leq 1-\d$ for some $\d>0$ fixed. This is by the concentration of $u$ around $Y$ equivalent to the assumption $Z\in E^\circ$, where $E$ is the elliptic set from the statement of the proposition.
Recall from the proof of Proposition \ref{prop_kernel_linearized} that the error bounds can be chosen uniform in $u$. Thus the same procedure can be repeated for the $u$-integral, giving
\begin{align}
 &\int_{\R/c_{a(t)}}
\exp\Big[-\frac{(a(t)-b)^2(a(t)+b)}{4a(t)b}\lb 
u+\frac{a(t)}{a(t)-b}\lb2Y-\frac{i(x_1-x_2)}{\sqrt{N}}+\frac{y_1+y_2}{\sqrt{N}}\rb\rb^2\Big]\notag\\
&\times \frac{1}{\sqrt{2\pi}}\sqrt{\frac{(a(t)-b)^2(a(t)+b)}{2a(t)b}} 1_{\big\{\frac{c_{a}^2b}{2aN}u^2+\frac{a(a-b)}{(a+b)N}X^
2\leq 1\big\}}\lb1+\O(1/\sqrt{N}\rb du\notag\\
&= 1_{\big\{\frac{a(a-b)}{(a+b)N}X^
2+\frac{a(a+b)}{(a-b)N}Y^
2\leq 1\big\}}+\O(1/\sqrt{N}).\label{indicator_strong}
\end{align}
To obtain the final form of the proposition, note that the determinant is invariant under conjugations of the kernel, i.e.~for any kernel $K$
\begin{align*}
\det(K(z_j,z_l))=\det(\tilde{K}(z_j,z_l)),
\end{align*}
where $\tilde{K}(z_j,z_l):=K(z_j,z_l)f(z_j)/f(\overline{z_l})$ and $f$ is some function without zeros or singularities. Using this, we see that the exponential factors 
\begin{align*}
\exp\lr\frac{i(a(t)-b)X(y_1-y_2) } { \sqrt{N}}-\frac{i(a(t)+b)Y(x_1-x_2)}{\sqrt{N}}
-ib\frac{(x_1y_1-x_2y_2)}{N}\rr
\end{align*}
cancel when taking the determinant.

Let us now consider the case $\t=0$. Then $b=0$ and it can be easily checked using polar coordinates that the orthonormal polynomials to the weight function $W_{a(t)}$ of Lemma \ref{lemma_weight} are
\begin{align*}
p_j(z):=\sqrt{\frac{a(t)^{j+1}}{\pi j!}}z^j,
\end{align*}
$\sqrt{\cdot}$ denoting the principal branch. This gives with \eqref{def_kernel}
\begin{align*}
K_{a(t)}(z_1,z_2)=\frac{a(t)}{\pi}\exp\lr -\frac{a(t)}2(\lv z_1\rv^2+\lv z_2\rv^2-2z_1\bar{z}_2)\rr Q(N,a(t)z_1\bar{z}_2).
\end{align*}
Invoking the asymptotics \eqref{eqn_indicator}, it is straightforward to finish the proof of the proposition.
\end{proof}

%%%%%%%%%%%%%%%%%%%%%%%%%%%%%%%%%%%%%%%%%%%%%%%%%%%%%%%%%%%%%%%%%
\section{Proof of main theorems for the trace-squared ensemble}\label{Sec-proof}
We will prove both main results simultaneously.
\begin{proof}[Proof of Theorem \ref{theorem_strong} and Theorem \ref{theorem_weak} for $\rho_{N,\Tr}^k$]\noindent
	
We will start with proving parts b) of both theorems. 
	Recall from \eqref{identity_corr_functions} 
	\begin{align}
		&\check{\rho}_{N,\tr}^{k}(\check{z})-\det(K_{\weak,\strong}(z_j,z_l))_{j,l\leq k}\notag\\
		&=\frac1{\sqrt{4\pi\g}}\int_\R \frac{\E_{a,b}e^{it 
				(\Tr \tilde{J}\tilde{J}^*-N(K_p+\hat{K}))}}{\E_{a,b}e^{-\g 
				(\Tr \tilde{J}\tilde{J}^*-N(K_p+\hat{K}))^2}}\lb\check{\rho}_{a(t),b}^{k}(\check{z})-\det(K_{\weak,\strong}(z_j,z_l))_{j,l\leq k}\rb 
e^{-\frac{t^2}{4\g}}dt,\label{identity_new}
	\end{align}
	where $K_{\weak,\strong}$ means eather $K_{\weak}$ or $K_{\strong}$ and $\hat{K}$ stands for $K$ in the case of Theorem \ref{theorem_strong} and $\bar{K}$ in the case of Theorem \ref{theorem_weak}. Furthermore, recall that $\check{\rho}_{N,\tr}^{k}(\check{z})$ and $\check{\rho}_{a(t),b}^{k}(\check{z})$  denote rescaled correlation functions of rescaled variables.
We have by Proposition \ref{prop_kernel_linearized} with $C:=C_{\bar{K}}$ and $\tilde{\a}:=C\a/\nu(X)$, or by Proposition 
\ref{strong_linearized} convergence of the term in the parenthesis to 0 with the error prescribed 
in the propositions, uniform for $\lv t\rv\leq 2\sqrt{\g\log N}$. Note here that the phase factor of the limiting kernel of 
Proposition \ref{prop_kernel_linearized} cancels when taking determinants. By \eqref{Laplace} and Lemma \ref{Lemma_K} we have that 
$\E_{a,b}e^{-\g 
	(\Tr \tilde{J}\tilde{J}^*-N(K_p+\hat{K}))^2}$ is bounded away from 0 uniformly in $N$. Clearly, $\lv\E_{a,b}e^{it 
	(\Tr \tilde{J}\tilde{J}^*-N(K_p+\hat{K}))}\rv$ is bounded above by 1. For $\lv t\rv>2\sqrt{\g\log N}$, we use 
\begin{align*}
	&\E_{a,b}\exp\lr{it 
		(\Tr \tilde{J}\tilde{J}^*-N(K_p+\hat{K}))}\rr P_{a(t),b}(J)=\exp\lr{-itN(K_p+\hat{K})}\rr P_{a,b}(J)\exp\lr{it 
		\Tr JJ^*}\rr
\end{align*}
and consequently
\begin{align}
	\lv \E_{a,b}e^{it 
		(\Tr \tilde{J}\tilde{J}^*-N(K_p+\hat{K}))}\check\rho_{a(t),b}^{k}\rv\leq\check\rho^k_{a,b}.\label{bound_corr_fct}
\end{align}
%where $\rho^k_{N,K,\t}$ denotes the $t$-independent correlation function of the eigenvalue density to $P_{N,K,\t}=P_{N,K,\t}^{t=0}$. 
The uniformity of the convergence to the bounded limiting kernel in Proposition \ref{prop_kernel_linearized} or Proposition \ref{strong_linearized} thus proves the boundedness of the integrand 
of \eqref{identity_new} in $t$ and $N$. Hence we can split up the $t$-integral into $\lv t\rv\leq 2\sqrt{\g\log N}$ and $\lv t\rv> 
2\sqrt{\g\log N}$. The first integral gives the desired result whereas the second integral is $\O(1/N)$ by \eqref{erfc_asymptotics1}.

For the remaining part of the proof, we concentrate on the more complicated case $\t\not=0$. Part a) of Theorem \ref{theorem_strong} follows for $Z\in E^\circ$ immediately from b). For $Z\notin E$, the statement follows from \eqref{indicator_strong} together with \eqref{identity_corr1} and \eqref{bound_corr_fct}.

To prove part a) of Theorem \ref{theorem_weak}, note that \eqref{density_weak_2} follows (formally) directly from b): Choosing $\a=\sqrt{2\kappa}\nu(X)$, $k=1$ and $z_1=z_2=iy$, we obtain
\begin{align*}
&\int\frac{1}{N}\rho_{N,\tr}^1(X+iY)dY=\int\frac{1}{N^2\nu(X)}\rho_{N,\tr}^1\lb X+\frac{iy}{N\nu(X)}\rb dy\\
&\to \nu(X) \frac{1}{2\pi}\int_{-\pi}^{\pi}\exp\lr-\frac{\a^2u^2}{2}\rr\int_\R\frac{\sqrt{2}}{\sqrt{\pi}\a}\exp\lr-\frac{
			2y^2}{ \a^2 }
		-2uy\rr dy
		du=\nu(X), \ N\to\infty.
\end{align*}
 To make this argument rigorous, we need to show interchangeability of limit and integration. In view of \eqref{identity_corr1}, Lemma \ref{Lemma_K} and \eqref{bound_corr_fct}, it suffices to show uniform integrability of $y\mapsto N^{-2}\rho_{a,b}^1(X+iy/N)$.
From \eqref{gamma1}, we get with $z:=X+iy/N$
\begin{align*}
	 K_{a}(z,z)&=\frac{\sqrt{a^2-b^2} \lv c_{a}\rv^2}{2\pi^2}\exp\lr-\frac{a(a+b)y^2}{bN^2}\rr \int_{\R}\exp\lr-\frac{(a+b)(a-b)^2}{4ab} u^2-\frac{2c_a^2uy}{N}\rr\notag\\
	 &\times\int_{\R}\exp\lr-\frac{(a+b)^2(a-b)}{4ab} v^2\rr Q\lb N,\frac{c_{a}^2b}{2a}(u^2-v^2)+\frac{a(a-b)}{a+b}X^2-i(a-b)vX\rb dvdu.
\end{align*}
As in the proof of Proposition \ref{prop_kernel_linearized}, $v$ can be assumed to be small. However, in contrast to that proof, equation \eqref{eqn_indicator} alone is not enough to see that $K_a(z,z)$ decays in $y$. We will use \eqref{eqn_indicator} as a bound for the incomplete gamma function for $\lv u\rv\leq M$, where $M>0$ is chosen such that 
\begin{align}
	\frac{c_{a}^2b}{2aN}M^2+\frac{a(a-b)}{(a+b)N}X^2>1+\e\label{def_M}
\end{align}
for some $\e>0$ and all $N$. For $\lv u\rv\leq M$ we have by \eqref{eqn_indicator} 
\begin{align*}
\lv Q\lb N,\frac{c_{a}^2b}{2a}(u^2-v^2)+\frac{a(a-b)}{a+b}X^2-i(a-b)vX\rb\rv=\O(1),
\end{align*}
where the $\O$ term is uniform in $u,v$ and $N$. For $\lv u\rv>M$, we use Lemma \ref{lemma_gamma} and \eqref{def_M} to get the bound
\begin{align*}
	\lv Q\lb N,\frac{c_{a}^2b}{2a}(u^2-v^2)+\frac{a(a-b)}{a+b}X^2-i(a-b)vX\rb\rv=\O(\exp\lr -Cu^2\rr)
\end{align*}
for some $C>0$ and where the $\O$-term is again uniform in $u,v$ and $N$. 
In total, this gives the bound
\begin{align*}
	 K_{a}(z,z)=\O\lb \exp\lr-\frac{a(a+b)y^2}{bN^2}\rr\int_{\R}\exp\lr-C'u^2-\frac{2c_a^2uy}{N}\rr du\rb,
\end{align*}
where $C'>\frac{(a+b)(a-b)^2}{4ab}$. As we have
\begin{align*}
	\int_{\R}\exp\lr-\frac{(a+b)(a-b)^2}{4ab}u^2-\frac{2c_a^2uy}{N}\rr du=\O\lb\exp\lr\frac{a(a+b)y^2}{bN^2}\rr\rb
\end{align*}
and $C'>\frac{(a+b)(a-b)^2}{4ab}$, we arrive at
\begin{align}
	K_{a}(z,z)=\O\lb\exp\lr -C''y^2\rr\rb\label{final_bound_K_a}
\end{align}
for some $C''>0$. Here we also used that $\frac{(a+b)(a-b)^2}{4ab}=\O(1)$ and $\frac{a(a+b)}{bN^2}=\O(1)$.
The bound \eqref{final_bound_K_a} clearly shows that $K_a$ is integrable in $y$ which completes the proof of \eqref{density_weak_2}.

Finally, to prove \eqref{density_weak_1}, it suffices, analogously to above, to consider $K_{a}(Z,Z)$, $Z=X+iY$, $Y\not=0$. From \eqref{gamma2}, we get
\begin{align*}
K_a(Z,Z)&=\frac{\sqrt{a^2-b^2} 
c_{a}^2}{2\pi^2}
\int_{\R}
\exp\Big[-\frac{(a-b)^2(a+b)}{4ab}\lb 
u+\frac{2a}{a-b} Y\rb^2\Big]\notag\\
&\times\int_{\R} 
\exp\Big[-\frac{(a+b)^2(a-b)}{4ab}\lb v-\frac{2ia}{a+b}X\rb^2\Big]\notag Q\lb N,\frac{c_{a(t)}^2b}{2a(t)}(u^2-v^2)\rb 
dvdu.
\end{align*}
In contrast to the strongly non-Hermitian situation, here the term $2a/(a-b)$ in front of $Y$ is of order $N$ which leads after a shift to
an expression $Q(N,C_N)$ with $C_N>0$ of order $N^2$. Now, \eqref{Temme1} and \eqref{erfc_asymptotics1} give the result.

\end{proof}

\section{Proof of main theorems for the elliptic fixed trace ensemble}\label{Sec_proof_FT}
The first aim is an explicit expression of the correlation functions of $P_{N,\FT}$ similar to \eqref{identity_corr1}. We start by adapting notation. Let $K_{\FT}=\lim_{\g\to\infty}\g K$ denote the limit whose existence has been shown in Lemma \ref{Lemma_K}. Define

\begin{align}\label{hat_a}
	\hat{a}(t):=N\lb\frac {1}{1-\t^2}+2K_{\text{\tiny FT}}\rb-it,\  \hat{a}:=\hat{a}(0)\ \text{ and recall }\ b=\frac{\t N}{1-\t^2}.
\end{align}
\begin{lemma}\label{Lemma_FT}
We have for any $\t\in(-1,1)$, $K_p>0$, $N\geq2$, $1\leq k\leq N-1$ and $z\in\C^k$
\begin{align}\label{identity_FT}
	\rho_{N,\FT}^k(z)=\frac1{C_{N,\FT}}\int_\R \E_{\hat{a},b}\exp\lr{it 
			(\Tr \tilde{J}\tilde{J}^*-NK_p)}\rr\rho_{\hat{a}(t),b}^{k}(z)dt,
\end{align}
where $C_{N,\FT}=C_{N,\FT}(k,\t,K_p)$ is a positive constant with the property
\begin{align}
	0<C_1\leq C_{N,\FT}\leq C_2<\infty \label{C_FT}
\end{align}
for all $N\geq2$ and some constants $C_1,C_2$ not depending on $N$ or $\t\in(-1,1)$. 
\end{lemma}
\begin{remark}
	$C_{N,\FT}$ is determined by the normalization condition
	\begin{align*}
		C_{N,\FT}=\frac{(N-k)!}{N!}\int_{\C^k}\int_\R \E_{\hat{a},b}\exp\lr{it 
			(\Tr \tilde{J}\tilde{J}^*-NK_p)}\rr\rho_{\hat{a}(t),b}^{k}(z)dtdz.
	\end{align*}
	A more constructive expression of $C_{N,\FT}$ will be given in the proof below (see \eqref{C_FT_constructive}).
\end{remark}

\begin{proof}[Proof of Lemma \ref{Lemma_FT}]
	By construction, we have for each $k$ and $N$
	\begin{align*}
		\lim_{\g\to\infty}\rho_{N,\tr}^k(dz)=\rho_{N,\FT}^k(dz)%\label{weak}
	\end{align*}
	in the weak sense. We will show in this proof that $\rho_{N,\tr}^k(z)$ converges, as $\g\to\infty$, pointwise in $z$ to a function given by the r.h.s.~of \eqref{identity_FT}. Using \eqref{characterization_corr_functions} and the dominated convergence theorem, \eqref{identity_FT} will follow.
	
Let us consider the r.h.s.~of \eqref{identity_corr1}, starting with 
	\begin{align}
		\sqrt{\g}\,\E_{a,b}\exp\lr{-\g 
			(\Tr \tilde{J}\tilde{J}^*-N(K_p+K))^2}\rr=\sqrt{\g}\int_0^\infty e^{-\g y}f_{N,\g}(y)dy,\label{Laplace2}
	\end{align}
	where $f_{N,\g}$ denotes the density of the distribution of $Y_N^2$ under $P_{a,b}$ and \\$Y_N:=\Tr 
\tilde{J}\tilde{J}^*-N(K_p+K)$. Defining $F_{N,\g}$ as the distribution function corresponding to $f_{N,\g}$, we find for $y\geq0$
	\begin{align*}
		&F_{N,\g}(y)=P_{a,b}(-\sqrt{y}\leq Y_N\leq \sqrt{y})=\tilde{F}_{N,\g}(\sqrt{y})-\tilde{F}_{N,\g}(-\sqrt{y}),
	\end{align*}
	$\tilde{F}_{N,\g}$ denoting the distribution function of $Y_N$. With the corresponding density $\tilde{f}_{N,\g}$ we see
	\begin{align}
		f_{N,\g}(y)=\frac{1}{\sqrt{y}}\cdot\frac12\lb \tilde{f}_{N,\g}(\sqrt{y})+\tilde{f}_{N,\g}(-\sqrt{y})\rb.\label{squareroot}
	\end{align}
	To compute the large $\g$ asymptotics of \eqref{Laplace2}, we need to determine the behavior of $f_{N,\g}$ and thus of $\tilde{f}_{N,\g}$ for $\g\to\infty$. To establish \eqref{identity_FT} for all $N$, we will first treat all terms $N$-independently. To study the large $N$ behavior in \eqref{C_FT}, we will then invoke asymptotic arguments (in $N$).

	Recall from \eqref{sum_rv} that $Y_N$ can be written as a positive linear combination of four independent, rescaled chi-squared distributed variables. The distribution function of a chi-squared distributed random variable with $n$ degrees of freedom is given by the normalized lower incomplete gamma function 
	\begin{align*}
		P\lb\frac n2,\frac y2\rb:=1-Q\lb\frac n2,\frac y2\rb.
	\end{align*}
	For example, for the first summand in \eqref{sum_rv},
	\begin{align}
		P_{a,b}\lb \bar{\l}_+^2\sqrt{2}\frac{Z_1-N(N-1)}{\sqrt{2N^2}}\leq \sqrt{y}\rb=P\lb \frac{N(N-1)}2,\frac{N(N-1)}2+\frac{N\sqrt{y}}{2\bar{\l}_+^2}\rb.\label{P_Gamma}
	\end{align}
The other three summands in \eqref{sum_rv} can be treated analogously. By the independence of the summands in \eqref{sum_rv} we find that $\tilde{f}_{N,\g}$ is the convolution of first derivatives of normalized incomplete Gamma functions. By the positivity of the density of chi-squared distributions and the fact that, because of the centering in $Z_1$ (and $Z_2$), $y=0$ corresponds to an interior point of the support of the density of the centered random variable, it is clear that $\tilde{f}_{N,\g}(0)>0$. $\tilde{f}_{N,\g}$ depends on $\g$ via the quantities $\bar\l_+,\bar\l_-,\s_{D,\Re}^2$ and $\s_{D,\Im}^2$, which all have (non-zero) limits by Lemma \ref{Lemma_K}  as $\g\to\infty$. It follows that $\tilde{f}_N(0)>0$, where $\tilde{f}_N(y):=\lim_{\g\to\infty}\tilde{f}_{N,\g}(y)$. Because derivatives (w.r.t.~the second variable) of the gamma function $P(w,y)$ in \eqref{P_Gamma} decay for large $y$, $\tilde{f}_{N,\g}$ actually converges uniformly in $y\in\R$ and is hence uniformly bounded.

 For the asymptotics of \eqref{Laplace2}, only the value $\tilde{f}_{N}(0)$ is important. Due to uniformity, it is assumed on the whole interval $(-\g^{-\e},\g^{-\e})$, $\e>0$ small enough. 
 Larger values of $y$ are irrelevant here because of the exponential decay of $\exp(-\g y)$. Hence the substitution $y'=\sqrt{y}$ gives with \eqref{squareroot}
 \begin{align}\label{C_FT_constructive}
 	\lim_{\g\to\infty}\sqrt{\g}\int_0^\infty e^{-\g y}f_{N,\g}(y)dy= \sqrt{\pi}\tilde{f}_N(0)=:C_{N,\FT}.
 \end{align}
Let us consider the other terms in \eqref{identity_corr1}. Clearly, $\exp[-t^2/(4\g)]$ converges to 1 for $\g\to\infty$. Because of the continuous dependence on $\g K$ (and the continuous dependence of $K$ on $\g$ itself by Lemma \ref{Lemma_K}), we have $\rho_{a(t),b}^k\to\rho_{\hat{a}(t),b}^k$ pointwise. From \eqref{def_kernel}, it follows easily that (for $N$ fixed) $\rho_{a(t),b}^k$ and $\rho_{\hat{a}(t),b}^k$ are bounded in $\g$ and $z$. Furthermore, in $t$, $\lv\rho_{a(t),b}^k(z)\rv\sim \lv t\rv^{kN}$. It remains to investigate $\E_{a,b}\exp\lr{it(\Tr \tilde{J}\tilde{J}^*-N(K_p+K))}\rr$. By the arguments above, it is the product of Fourier transforms of four, partially centered, chi-squared distributions with in total $2N^2$ degrees of freedom. Thus $\E_{a,b}\exp\lr{it(\Tr \tilde{J}\tilde{J}^*-N(K_p+K))}\rr$ decays in $t$ like $\lv t\rv^{-N^2}$.  Recall that $k<N$ and $N\geq2$. We find that the decay of $\E_{a,b}\exp\lr{it(\Tr \tilde{J}\tilde{J}^*-N(K_p+K))}\rr$ offsets the increase in $t$ of the correlation functions.
By Lemma \ref{Lemma_K}, $K$ converges to 0 as $\g\to\infty$ and as before, $a\to\hat{a}$. Employing arguments analogous to those discussing the uniform convergence of $\tilde{f}_{N,\g}$ to $\tilde{f}_N$, we find that $\E_{a,b}\exp\lr{it(\Tr \tilde{J}\tilde{J}^*-N(K_p+K))}\rr$ converges uniformly in $t$ to  $\E_{\hat{a},b}\exp\lr{it(\Tr \tilde{J}\tilde{J}^*-NK_p)}\rr$ as $\g\to\infty$. Keeping the decay in mind, we can apply the dominated convergence theorem and interchange the limit $\g\to\infty$ with the $t$-integral, thereby proving \eqref{identity_FT}.

To establish \eqref{C_FT}, by \eqref{C_FT_constructive} we need to study the large $N$ asymptotics of $\tilde{f}_N(0)$. Recall that $\tilde{f}_N$ is the convolution of first derivatives of incomplete Gamma functions. Let us again exemplarily look at \eqref{P_Gamma} and invoke Lemma \ref{lemma_gamma} for its large $N$ behavior. In the notation of Lemma \ref{lemma_gamma}, we have 
	\begin{align*}
		\eta=\frac{\sqrt{y}}{{\hat{\l}_+^2}(N-1)}+\O\lb\frac1{N^{3/2}}\rb%\label{asymp_eta},
	\end{align*}
	where $\hat{\l}_+:=\lim_{\g\to\infty}\bar{\l}_+$. Thus $\eta\sqrt{N(N-1)/4}$ is of order 1 in $N$.
Since Lemma \ref{lemma_gamma} shows the convergence of $P(w,wz)$ (as $w\to\infty$) to the error function $\textup{erf}:=1-\erfc$ to be uniform in $z$ in (sufficiently large) complex domains and $P(w,\cdot)$ is an analytic function (in the second variable), the convergence extends to all derivatives by Cauchy's integral formula. Hence the density of the first summand in \eqref{sum_rv} converges to a Gaussian density as $N\to\infty$, uniformly on $\R$. Similar convergence holds for the three other summands, where those corresponding to $Z_3$ and $Z_4$ play the role of delta functions in the limit. We conclude that $\tilde{f}_N$ converges uniformly in $y$ to a Gaussian density as well. The uniformity in $\t$ can be seen from \eqref{variance_diag}, \eqref{variance_off} and \eqref{cov_real}. This proves \eqref{C_FT} and thus the lemma.

\end{proof}

\begin{proof}[Proof of Theorem \ref{theorem_strong} and Theorem \ref{theorem_weak} for $\rho_{N,\FT}^k$]\noindent
	
	Parts (a) of the theorems follow from parts (b) as in the case of the trace-squared ensemble, see Section \ref{Sec-proof}. We will again deal with weak and strong non-Hermiticity simultaneously.

Starting from \eqref{identity_FT}, we study
\begin{align}
	&\check\rho_{N,\FT}^k(\check z)-\det(K_{\strong/\weak}(z_j,z_l))_{j,l\leq k}\notag\\
	&=\frac1{C_{N,\FT}}\int_\R \E_{\hat{a},b}\exp\lr{it 
		(\Tr \tilde{J}\tilde{J}^*-NK_p)}\rr\lb\check\rho_{\hat{a}(t),b}^{k}(\check 
z)-\det(K_{\strong/\weak}(z_j,z_l))_{j,l\leq k}\rb dt,\label{identity_FT2}
\end{align}
$\check\rho_{N,\FT}^k$ and $\check z$ indicating appropriate rescaling like for the trace-squared ensemble (see 
\eqref{identity_corr_functions} and \eqref{rescaling}). For the limit of weak non-Hermiticity, the constant $K_\FT$ in the 
definition of $\hat{a}$ in \eqref{hat_a} has to be replaced by $\bar{K}_{\textup{FT}}$, which has been defined in Lemma 
\ref{Lemma_K}.
By Lemma \ref{Lemma_FT}, the constant $C_{N,\FT}$ has no influence on the possible convergence of \eqref{identity_FT2}. The difference of correlation functions in the integral converges to 0 uniformly for $\lv t\rv=\O(\sqrt{\log N})$ by Propositions \ref{prop_kernel_linearized} and  \ref{strong_linearized}. This will not be the case for $t$ being large in comparison to $N$. In fact, $\lv\check\rho_{\hat{a}(t),b}^k(\check z)\rv$ will increase as $\lv t\rv^{kN}$ for $N$ fixed, as seen in the proof of Lemma \ref{Lemma_FT}. In that proof, we also saw that $\lv \E_{\hat{a},b}\exp\lr{it 
		(\Tr \tilde{J}\tilde{J}^*-NK_p)}\rr\rv$ has a stronger decay in $t$ for $N$ fixed. Here, it is our task to show a version of this statement which is uniform in $N$.

To this end, let us consider first
\begin{align*}
	\E_{\hat{a},b}\exp\lr{it(\Tr \tilde{J}\tilde{J}^*-NK_p)}\rr=\int_\R e^{ity}\tilde{f}_N(y)dy.
\end{align*}
Recall that $\Tr \tilde{J}\tilde{J}^*-NK_p$ is the sum of four (partially centered) indepedent chi-squared random variables. Exemplarily,
\begin{align}
	&\lv \E_{\hat{a},b}\exp\lr it\sqrt{2}\bar{\l}^2_+\lb \frac{Z_1-N(N-1)}{\sqrt{2N^2}}\rb\rr\rv=\lv \lb1-2i\frac{\bar{\l}^2_+t}{\sqrt{N^2}}\rb^{-\frac{N(N-1)}2}\exp\lr-i\bar{\l}^2_+t\frac{N(N-1)}{\sqrt{N^2}}\rr\rv\notag\\
	&=\lb 1+\frac{(\bar{\l}^2_+t)^2}{N^2/2}\rb^{-N(N-1)/4}=\exp\lb-(N(N-1)/4)\log\lb 1+\frac{(\bar{\l}^2_+t)^2}{N^2/2}\rb\rb.\label{decay5}
\end{align}

%The latter converges for $t$ fixed and $N\to\infty$ to $\exp\lr -(\bar{\l}^2_+t)^2/2\rr$.
%By the uniform convergence of $\rho_{\hat{a}(t),b}^k$ for $\lv t\rv=\O(\sqrt{\log N})$, it suffices to consider 

Let us now consider $\check\rho_{\hat{a}(t),b}^k(\check z)$. It is instructive to look first at the simpler case $\t=0$ as the more 
complicated case $\t\not=0$ will be partly reduced to this situation. $\rho_{\hat{a}(t),b}^k(z)$ is (up to constants) 
the sum of $k!$ summands, each being a $k$-fold product of terms of the form
\begin{align*}
K_{\hat{a}(t)}(z_j,z_l)=\frac{\hat{a}(t)}{\pi}\exp\lr -\frac{\hat{a}(t)}2(\lv z_j\rv^2+\lv z_l\rv^2-2z_j\bar{z}_l)\rr Q(N,\hat{a}(t)z_j\bar{z}_l), 
\end{align*}
thus it suffices to bound $K_{\hat{a}(t)}$. Here, $z_j,z_l=Z+\O(1/N)$.  From Lemma \ref{lemma_gamma}, \eqref{erfc_asymptotics1} and \eqref{erfc_asymptotics2}, we see that $\lv K_{\hat{a}(t)}(z_j,z_l)\rv$ gets large only if $N=\O(\lv t\rv)$ in which case it might grow as $\exp(\O(N\log\lv t/N\rv))$. Such a growth is suppressed by the decay in \eqref{decay5}. 

Now assume $\t\not=0$. We will show first that \eqref{v-integral} (with $\hat{a}(t)$ instead of $a(t)$) can effectively be replaced by 
\begin{align}
	\sqrt{\frac{4\pi \hat{a}(t)b}{(\hat{a}(t)+b)^2(\hat{a}(t)-b)}}Q\lb N,\frac{c_{\hat{a}(t)}^2b}{2\hat{a}(t)}u^2+\frac{c_{\hat{a}(t)}^2b(z_1+\overline{z_2})^2}{2\hat{a}(t)(1+b/\hat{a}(t))^2}\rb,\label{saddle_claim}
\end{align}
independently of the size of $t$. To establish this, we have to see that the possible growth of 
\begin{align}
	\exp\lr -\frac{c_{\hat{a}(t)}^2b}{2\hat{a}(t)}(u^2-v^2)+i\frac{c_{\hat{a}(t)}^2b(z_1+\overline{z_2})}{\hat{a}(t)+b}v-\frac{c_{\hat{a}(t)}^2b(z_1+\overline{z_2})^2}{2\hat{a}(t)(1+b/\hat{a}(t))^2}\rr,\label{FT_growth1}
\end{align}
as $v\to\infty$, is negligible compared to the decay of 
\begin{align}
\exp\lr -\frac{(\hat{a}(t)+b)^2(\hat{a}(t)-b)}{4\hat{a}(t)b}v^2\rr.\label{FT_decay}
\end{align}
The decisive term in \eqref{FT_growth1} is as before 
\begin{align}
\exp\lr \frac{c_{\hat{a}(t)}^2b}{2\hat{a}(t)}v^2\rr.\label{FT_growth2}
\end{align}
We saw in \eqref{decay1} and \eqref{decay2} that for the trace-squared ensemble, the exponent of \eqref{FT_decay} was larger than the one of \eqref{FT_growth2}. Since
\begin{align*}
\frac{(\hat{a}(t)+b)^2(\hat{a}(t)-b)}{4\hat{a}(t)b}=\frac{c_{\hat{a}(t)}^2b}{2\hat{a}(t)}\cdot\lb1+\frac{\hat{a}(t)}{b}\rb
\end{align*}
and $\lv1+\hat{a}(t)/b\rv>1$ for any $N,t$, this remains to be true for the fixed trace ensemble.
Recall that the contour of integration is $\R/c_{\hat{a}(t)}+C_{z_1,z_2}$, where (cf.~\eqref{def_C_z_1})
\begin{align}
C_{z_1,z_2}=i(z_1+\overline{z_2})\frac{\hat{a}(t)}{\hat{a}(t)+b}.\label{def_C}
\end{align} 
It is a straightforward, but cumbersome task to compute the large $N$, large $t$ asymptotics of the involved terms. For instance, 
in the situation of strong non-Hermiticity the asymptotics of $\hat{a}(t)/(\hat{a}(t)+b)$ read 
\begin{align*}
 \lim_{N,t\to\infty}\frac{\hat{a}(t)}{\hat{a}(t)+b}=\begin{cases}
                                 \frac{1+2K_{\textup{FT}}(1-\t^2)}{1+\t+2K_{\textup{FT}}(1-\t^2)},&\ \text{ if  }t=o(N),\\
				 \frac{1+2K_{\textup{FT}}(1-\t^2)-is}{1+\t+2K_{\textup{FT}}(1-\t^2)-is},&\ \text{ if  }t=sN,\\
				 1,&\ \text{ if }t\gg N.
				\end{cases}
\end{align*}
Here $K_{\textup{FT}}$ has been defined in Lemma \ref{Lemma_K}. In the situation of weak non-Hermiticity, we have
\begin{align*}
 \lim_{N,t\to\infty}\frac{\hat{a}(t)}{\hat{a}(t)+b}=\begin{cases}
				 \frac12,&\ \text{ if }t=o(N^2),\\
				 \frac{1-ir}{2-ir},&\ \text{ if }t=rN^2,\\
				 1,&\ \text{ if }t\gg N^2.	
                                \end{cases}
\end{align*}
 It follows by \eqref{def_C} that $C_{z_1,z_2}$ is always 
$\O(1)$. This allows to shift the contour such that it passes through the origin.

Note that since \eqref{gamma1} originates from \eqref{kernel1} and $C_{z_1,z_2}=\O(1)$, it is clear 
that \eqref{FT_decay} decays on $\R/c_{\hat{a}(t)}$.
To see that the integral for the fixed trace analog of \eqref{v-integral} concentrates on $v=0$, it is necessary to know that the exponent of \eqref{FT_decay} goes to infinity as $N,t\to\infty$. To this end, we note firstly
\begin{align*}
\frac{(\hat{a}(t)+b)^2(\hat{a}(t)-b)}{4\hat{a}(t)b}=c_{\hat{a}(t)}^2\frac{\hat{a}(t)}{2(\hat{a}(t)+b)}
\end{align*}
and secondly the asymptotics for $c_{a(t)}^2$ as $N,t\to\infty$. In the situation of strong non-Hermiticity
\begin{align*}
 c_{\hat{a}(t)}^2\sim\begin{cases}
                                 \frac{N}{2\t}(1+4K_{\textup{FT}}+4K_{\textup{FT}}^2(1-\t^2)),&\ \text{ if }t=o(N),\\
\frac{N}{2\t}(1+4K_{\textup{FT}}+4K_{\textup{FT}}^2(1-\t^2))-\frac{isN}\t(1+2K_{\textup{FT}}(1-\t^2))-\frac{s^2N(1-\t^2)}{2\t},&\ 
\text{ if }t=sN,\\
				 -\frac{t^2(1-\t^2)}{2\t N},&\ \text{ if }t\gg N,
				\end{cases}
\end{align*}
and in the situation of weak non-Hermiticity
\begin{align*}
c_{\hat{a}(t)}^2\sim\begin{cases}
				 \frac {NC_{\bar{K}_{\textup{FT}}}}2,&\ \text{ if  }t=o(N),\\
				 \frac {NC_{\bar{K}_{\textup{FT}}}}2-isN,&\ \text{ if  }t=sN,\\
				 -it,&\ \text{ if  }N\ll t\ll N^2,\\
					-irN^2-\frac{rN^2\tilde{\a}^2}{2C_{\bar{K}_{\textup{FT}}}^2},&\ \text{ if  }t=rN^2,\\
					-\frac{t^2\tilde{\a}^2}{2C_{\bar{K}_{\textup{FT}}}^2N^2},&\ \text{ if }t\gg N^2,
                                \end{cases}
\end{align*}
where $C_{\bar{K}_{\textup{FT}}}:=1+4\bar{K}_{\textup{FT}}$. 
From here it is straightforward to check \eqref{saddle_claim} analogously to the proofs of Propositions \ref{prop_kernel_linearized} 
and \ref{strong_linearized}. 

 For the $u$-integral in the fixed trace analog of \eqref{gamma1} we have to distinguish two cases. In the situation of weak 
non-Hermiticity and $t=\O(N)$, the integral can be treated as in the proof of Proposition \ref{prop_kernel_linearized} and the 
result of Proposition \ref{prop_kernel_linearized} holds with error $\O(1)$ instead of $\O(\log N/\sqrt N)$. In all other 
situations, the $u$-integral can be treated analogously to the $v$-integral yielding
\begin{align*}
&K_{\hat{a}(t)}(z_1,z_2)\sim\frac{\hat{a}(t)b}{\pi}\frac{\lv \hat{a}(t)+b\rv\lv \hat{a}(t)-b\rv}{\sqrt{(\hat{a}(t)+b)(\hat{a}(t)-b)}}\notag\\
 		&\times\exp\lr-\frac {\hat a(t)}2(\lv z_1\rv^2+\lv 
 z_2\rv^2)+\frac {ib}2\lb\Im(z_2^2)-\Im(z_1^2)\rb+\hat a(t)z_1\overline{z_2}\rr\\
&\times Q\lb N,\frac{\hat{a}(t)^2b}{\hat{a}(t)^2-b^2}(z_1^2+\overline{z_2}^2)+\frac{\hat{a}(t)(\hat{a}(t)^2+b^2)}{\hat{a}(t)^2-b^2}z_1\overline{z_2}\rb.
\end{align*}
From here, the same arguments as for the case $\t=0$ (see the paragraph following \eqref{decay5}) can be invoked. It follows that the function 
\begin{align}\label{unif_int}
t\mapsto \E_{\hat{a},b}\exp\lr{it 
		(\Tr \tilde{J}\tilde{J}^*-NK_p)}\rr\lb\check\rho_{\hat{a}(t),b}^{k}(\check 
z)-\det(K_{\strong/\weak}(z_j,z_l))_{j,l\leq k}\rb
\end{align}
	is uniformly integrable in $N$, hence limit and integration can be interchanged. To see the claimed rates of convergence, 
we use Propositions 
\ref{prop_kernel_linearized} and \ref{strong_linearized} for $\lv t\rv=\O(\sqrt{\log N})$. For $\lv t\rv=o(N)$, \eqref{decay5} shows that \eqref{unif_int} has sub-Gaussian decay and for even larger $\lv t\rv$ the polynomial decay is sufficient.
\end{proof}

\bibliographystyle{plain}
\bibliography{bibliography.bib}

\end{document}